\documentclass[10pt]{amsart}
\usepackage{amsfonts}
\usepackage{mathrsfs}
\usepackage{amscd}
\usepackage{amsmath}
\usepackage{amssymb}
\usepackage{latexsym}
\usepackage{lscape}
\usepackage{xypic}
\usepackage{comment}
\usepackage{amscd}
\usepackage{wasysym}  
\usepackage{tikz} 
\usetikzlibrary{matrix,arrows}
\usepackage{appendix}
\usepackage{bibtopic}
\usepackage{geometry}

\newtheorem{thm}{Theorem}[section]

\newtheorem{prop}[thm]{Proposition}
\newtheorem{lem}[thm]{Lemma}

\newtheorem{lem-def}[thm]{Lemma-Definition}
\newtheorem{cor}[thm]{Corollary}

\theoremstyle{definition}

\newtheorem{ex}[thm]{Example}
\newtheorem{construction}[thm]{Construction}

\newtheorem{rmk}[thm]{Remark}
\newtheorem{dfn}[thm]{Definition}

\numberwithin{equation}{section}

\newcommand{\nc}{\newcommand}

\nc{\on}{\operatorname}
\nc{\fraka}{{\mathfrak a}}
\nc{\frakb}{{\mathfrak b}}
\nc{\frakc}{{\mathfrak c}}
\nc{\frakd}{{\mathfrak d}}
\nc{\frake}{{\mathfrak e}}
\nc{\frakf}{{\mathfrak f}}
\nc{\frakg}{{\mathfrak g}}
\nc{\frakh}{{\mathfrak h}}
\nc{\fraki}{{\mathfrak i}}
\nc{\frakj}{{\mathfrak j}}
\nc{\frakk}{{\mathfrak k}}
\nc{\frakl}{{\mathfrak l}}
\nc{\frakm}{{\mathfrak m}}
\nc{\frakn}{{\mathfrak n}}
\nc{\frako}{{\mathfrak o}}
\nc{\frakp}{{\mathfrak p}}
\nc{\frakq}{{\mathfrak q}}
\nc{\frakr}{{\mathfrak r}}
\nc{\fraks}{{\mathfrak s}}
\nc{\frakt}{{\mathfrak t}}
\nc{\fraku}{{\mathfrak u}}
\nc{\frakv}{{\mathfrak v}}
\nc{\frakw}{{\mathfrak w}}
\nc{\frakx}{{\mathfrak x}}
\nc{\fraky}{{\mathfrak y}}
\nc{\frakz}{{\mathfrak z}}
\nc{\frakA}{{\mathfrak A}}
\nc{\frakB}{{\mathfrak B}}
\nc{\frakC}{{\mathfrak C}}
\nc{\frakD}{{\mathfrak D}}
\nc{\frakE}{{\mathfrak E}}
\nc{\frakF}{{\mathfrak F}}
\nc{\frakG}{{\mathfrak G}}
\nc{\frakH}{{\mathfrak H}}
\nc{\frakI}{{\mathfrak I}}
\nc{\frakJ}{{\mathfrak J}}
\nc{\frakK}{{\mathfrak K}}
\nc{\frakL}{{\mathfrak L}}
\nc{\frakM}{{\mathfrak M}}
\nc{\frakN}{{\mathfrak N}}
\nc{\frakO}{{\mathfrak O}}
\nc{\frakP}{{\mathfrak P}}
\nc{\frakQ}{{\mathfrak Q}}
\nc{\frakR}{{\mathfrak R}}
\nc{\frakS}{{\mathfrak S}}
\nc{\frakT}{{\mathfrak T}}
\nc{\frakU}{{\mathfrak U}}
\nc{\frakV}{{\mathfrak V}}
\nc{\frakW}{{\mathfrak W}}
\nc{\frakX}{{\mathfrak X}}
\nc{\frakY}{{\mathfrak Y}}
\nc{\frakZ}{{\mathfrak Z}}
\nc{\bbA}{{\mathbb A}}
\nc{\bbB}{{\mathbb B}}
\nc{\bbC}{{\mathbb C}}
\nc{\bbD}{{\mathbb D}}
\nc{\bbE}{{\mathbb E}}
\nc{\bbF}{{\mathbb F}}
\nc{\bbG}{{\mathbb G}}
\nc{\bbH}{{\mathbb H}}
\nc{\bbI}{{\mathbb I}}
\nc{\bbJ}{{\mathbb J}}
\nc{\bbK}{{\mathbb K}}
\nc{\bbL}{{\mathbb L}}
\nc{\bbM}{{\mathbb M}}
\nc{\bbN}{{\mathbb N}}
\nc{\bbO}{{\mathbb O}}
\nc{\bbP}{{\mathbb P}}
\nc{\bbQ}{{\mathbb Q}}
\nc{\bbR}{{\mathbb R}}
\nc{\bbS}{{\mathbb S}}
\nc{\bbT}{{\mathbb T}}
\nc{\bbU}{{\mathbb U}}
\nc{\bbV}{{\mathbb V}}
\nc{\bbW}{{\mathbb W}}
\nc{\bbX}{{\mathbb X}}
\nc{\bbY}{{\mathbb Y}}
\nc{\bbZ}{{\mathbb Z}}
\nc{\calA}{{\mathcal A}}
\nc{\calB}{{\mathcal B}}
\nc{\calC}{{\mathcal C}}
\nc{\calD}{{\mathcal D}}
\nc{\calE}{{\mathcal E}}
\nc{\calF}{{\mathcal F}}
\nc{\calG}{{\mathcal G}}
\nc{\calH}{{\mathcal H}}
\nc{\calI}{{\mathcal I}}
\nc{\calJ}{{\mathcal J}}
\nc{\calK}{{\mathcal K}}
\nc{\calL}{{\mathcal L}}
\nc{\calM}{{\mathcal M}}
\nc{\calN}{{\mathcal N}}
\nc{\calO}{{\mathcal O}}
\nc{\calP}{{\mathcal P}}
\nc{\calQ}{{\mathcal Q}}
\nc{\calR}{{\mathcal R}}
\nc{\calS}{{\mathcal S}}
\nc{\calT}{{\mathcal T}}
\nc{\calU}{{\mathcal U}}
\nc{\calV}{{\mathcal V}}
\nc{\calW}{{\mathcal W}}
\nc{\calX}{{\mathcal X}}
\nc{\calY}{{\mathcal Y}}
\nc{\calZ}{{\mathcal Z}}

\nc{\scrA}{{\mathscr A}}
\nc{\scrB}{{\mathscr B}}
\nc{\scrR}{{\mathscr R}}

\nc{\bnu}{{\bar{ \nu}}}
\nc{\Gmon}{{\bbG_m\text{-mon}}}

\nc{\olO}{\bar{\calO}}

\nc{\al}{{\alpha}} 
\nc{\be}{{\beta}}
\nc{\ga}{{\gamma}} \nc{\Ga}{{\Gamma}}
 \nc{\hGa}{\hat{\Gamma}}
\nc{\ve}{{\varepsilon}} 
\nc{\la}{{\lambda}} \nc{\La}{{\Lambda}}
\nc{\om}{\omega} \nc{\Om}{\Omega} 
\nc{\sig}{{\sigma}} \nc{\Sig}{{\Sigma}}

\nc{\tnb}{\psi_{\rm tame}}

\nc{\Eq}{\text{Eq}} 
\nc{\colim}{\text{colim}} 
\nc{\op}{{\on{op}}}
\nc{\ad}{{\on{ad}}}
\nc{\alg}{{\on{alg}}}
\nc{\Ad}{{\on{Ad}}}
\nc{\Adm}{{\on{Adm}}} \nc{\aff}{{\on{af}}}
\nc{\Aut}{{\on{Aut}}}
\nc{\Bun}{{\on{Bun}}}
\nc{\cha}{{\on{char}}}
\nc{\der}{{\on{der}}}
\nc{\Der}{{\on{Der}}}
\nc{\diag}{{\on{diag}}}
\nc{\End}{{\on{End}}}
\nc{\Fl}{{\calF\!\ell}}
\nc{\Gal}{{\on{Gal}}}
\nc{\Gr}{{\on{Gr}}}
\nc{\rH}{{\on{H}}}
\nc{\Hom}{{\on{Hom}}}
\nc{\IC}{{\on{IC}}}
\nc{\id}{{\on{id}}}
\nc{\Id}{{\on{Id}}}
\nc{\ind}{{\on{ind}}}
\nc{\Ind}{{\on{Ind}}}
\nc{\Lie}{{\on{Lie}}}
\nc{\Pic}{{\on{Pic}}}
\nc{\pr}{{\on{pr}}}
\nc{\Res}{{\on{Res}}}
\nc{\res}{{\on{res}}} \nc{\Sat}{{\on{Sat}}}
\nc{\s}{{\on{sc}}}
\nc{\drv}{{\on{der}}}
\nc{\sgn}{{\on{sgn}}}
\nc{\Spec}{{\on{Spec}}}\nc{\Spf}{\on{Spf}} 
\nc{\Sph}{\on{Sph}}
\nc{\St}{{\on{St}}}
\nc{\tr}{{\on{tr}}}
\nc{\Tr}{{\on{Tr}}}
\nc{\Mod}{{\mathrm{-Mod}}}
\nc{\Hilb}{{\on{Hilb}}} 
\nc{\Ext}{{\on{Ext}}} 
\nc{\vs}{{\on{Vec}}}
\nc{\ev}{{\on{ev}}}
\nc{\nO}{{\breve{\calO}}}
\nc{\tS}{{\tilde{S}}}
\nc{\spe}{{\on{sp}}}
\nc{\cok}{{\on{Coker}}}

\nc{\nscrR}{{\mathscr{R}^{\on{nr}}}}

\nc{\GL}{{\on{GL}}}
\nc{\U}{{\on{U}}}
\nc{\Gl}{\on{Gl}} 
\nc{\GSp}{{\on{GSp}}}
\nc{\gl}{{\frakg\frakl}}
\nc{\SL}{{\on{SL}}} 
\nc{\SU}{{\on{SU}}} 
\nc{\SO}{{\on{SO}}}

\nc{\Conv}{{\on{Conv}}}
\nc{\Rep}{{\on{Rep}}}
\nc{\Dom}{{\on{Dom}}}
\nc{\red}{{\on{red}}}
\nc{\act}{{\on{act}}}
\nc{\nr}{{\on{nr}}}
\nc{\ctf}{{\on{ctf}}}

\nc{\str}{{\on{-}}} 
\nc{\os}{{\bar{s}}}
\nc{\oeta}{{\bar{\eta}}}

\nc{\hookto}{\hookrightarrow}
\nc{\longto}{\longrightarrow}
\nc{\leftto}{\leftarrow}
\nc{\onto}{\twoheadrightarrow}
\nc{\lonto}{\twoheadleftarrow}

\nc{\bio}{{\bar{i}}}
\nc{\bjay}{{\bar{j}}}

\nc{\oFl}{{\overline{\Fl}}} 
\nc{\bU}{{\overline{U}}}
\nc{\tGr}{{\tilde{\Gr}}}
\nc{\cGr}{\calG\! r}
\nc{\oGr}{\overline{\on{Gr}}} 
\nc{\ocGr}{\overline{\calG\! r}}
\nc{\co}{{\colon}}
\nc{\sch}[1]{(Sch/{#1})}
\nc{\HypLoc}[1]{HypLoc({#1})}

\nc{\ohtimes}{\stackrel{!}{\otimes}}
\nc{\boxtilde}{\widetilde{\boxtimes}}
\nc{\vstar}{{\varhexstar}}

\nc{\bslash}{\backslash}
\nc{\lisset}{\text{lis-\'et}}
\nc{\algQl}{{\bar{\bbQ}_\ell}}
\nc{\sF}{{\bar{F}}}
\nc{\nF}{{\breve{F}}}
\nc{\nW}{{W^{\on{nr}}}}
\nc{\sk}{{\bar{k}}}
\nc{\cont}{\on{c}}
\nc{\supp}{\on{supp}}
\nc{\blt}{\bullet}  
\nc{\dom}{\on{dom}}
\nc{\scon}{{\on{sc}}} 
\nc{\Affine}{\on{Aff}} 
\nc{\nscrA}{\mathscr{A}^{\on{nr}}} 
\nc{\nfraka}{{\fraka^{\on{nr}}}}
\nc{\ran}{{\rangle}}
\nc{\lan}{{\langle}}
\nc{\bk}{{\bar{k}}}
\nc{\tF}{{\tilde{F}}}
\nc{\sS}{{\bar{S}}}
\nc{\LG}{{^\text{L}\hspace{-0.04cm}G}}
\nc{\LL}{{^\text{L}\hspace{-0.07cm}L}}

\nc{\pot}[1]{ [\hspace{-0,5mm}[ {#1} ]\hspace{-0,5mm}] }
\nc{\rpot}[1]{ (\hspace{-0,7mm}( {#1} )\hspace{-0,7mm}) }

\nc{\defined}{\hspace{0.1cm}\stackrel{\text{\tiny \rm def}}{=}\hspace{0.1cm}}

\begin{document}

\title[Hyperbolic localization and nearby cycles]{Spaces with $\bbG_m$-action, hyperbolic localization \\ and nearby cycles}
\author[Timo Richarz]{by Timo Richarz$^*$}
\thanks{$^*$This work was finished while the author was supported by the Max-Planck-Institut f\"ur Mathematik in Bonn. He thanks everyone cordially for hospitality and excellent working conditions.}

\address{Timo Richarz, University of Duisburg-Essen, Faculty of Mathematics, Thea-Leymann-Str. 9, 45127 Essen, Germany}
\email{timo.richarz@uni-due.de}

\maketitle

\begin{abstract}
We study families of algebraic spaces with $\bbG_m$-action and prove Braden's theorem \cite{Br03}, \cite{DG15} on hyperbolic localization for arbitrary base schemes. As an application, we obtain that hyperbolic localization commutes with nearby cycles.
\end{abstract}

\setcounter{section}{-1}


\phantom{h}

\thispagestyle{empty}

\section{Introduction} 
Algebraic varieties $X$ with an action of the multiplicative group $\bbG_m$ are a classical object of study \cite{BB73}. The $\bbG_m$-action induces two stratifications on $X$: the strata of points $X^+$ floating to the fixed points, and the strata of points $X^-$ floating away from the fixed points. 
In \cite{Br03} Braden proves a general theorem on localizing equivariant objects on $X$ to the subspace of fixed points $X^0$. The result is used in a number of places \cite{MV07}, \cite{Ach11}, \cite{AcHR15}, and has proven to be of importance for geometric methods in representation theory, e.g. induction and restriction of character sheaves. Braden's theorem is generalized by Drinfeld and Gaitsgory \cite{DG15} to algebraic spaces over fields. In the present manuscript, we consider families of spaces with $\bbG_m$-action and study the behavior under base change. The main motivation is the commutation of hyperbolic localization with nearby cycles which is inspired by a result of Arkhipov-Bezrukavnikov \cite[Thm. 4]{AB09}, and which is used in subsequent work to prove the test function conjecture of Haines-Kottwitz for parahoric local models, cf.~\cite{HR18a,HR18b}. 


\subsection{Statement of results} Let $S$ be a scheme, and let $X/S$ be an algebraic space in the sense of \cite{StaPro}. If $\bbG_m$ acts on $X/S$ (trivial on $S$), there are the following three functors on the category of $S$-schemes
\begin{equation}\label{intro1}
\begin{aligned}
\hspace{1cm} X^0\co&\; T\longmapsto \Hom^{\bbG_{m}}_S(T, X)\\
\hspace{1cm} X^+\co& \; T\longmapsto \Hom^{\bbG_{m}}_S((\bbA_T^1)^+, X)\\
\hspace{1cm}X^-\co& \; T\longmapsto \Hom^{\bbG_{m}}_S((\bbA^1_T)^-, X),
\end{aligned}
\end{equation}
where $(\bbA_T^1)^+$ (resp. $(\bbA_T^1)^-$) is $\bbA^1_T$ with the usual (resp. opposite) $\bbG_m$-action. The functor $X^0$ is the functor of $\bbG_m$-fixed points in $X$, and $X^+$ (resp. $X^-$) is called the attractor (resp. repeller). Informally speaking $X^+$ (resp. $X^-$) is the space of points $x$ such that the limit $\lim_{\la\to 0}\la\cdot x$ (resp. $\lim_{\la\to \infty}\la\cdot x$) exists. Note that the formation of $X^0$ and $X^\pm$ commutes with arbitrary base change $S'\to S$.\\
In many cases, the $\bbG_m$-action on a space is locally linear, and we consider the following notion. We say that a $\bbG_m$-action on $X/S$ is \emph{\'etale locally linearizable} if there exists a $\bbG_m$-equivariant \'etale covering family $\{U_i\to X\}_i$, where the $U_i$ are $S$-affine schemes with $\bbG_m$-action. By upcoming results of Alper-Hall-Rydh \cite{AHR16}, cf. \S \ref{secintro2} for more details, every $\bbG_m$-action on a quasi-separated algebraic space $X/S$ locally of finite presentation is \'etale locally linearizable (no condition on $S$). \smallskip\\

\noindent {\bf Theorem A.} {\it Let $S$ be a scheme, and let $X/S$ be an algebraic space with an \'etale locally linearizable $\bbG_m$-action.\smallskip\\
i) The functor $X^0$ is representable by a closed subspace of $X$. \smallskip\\
ii) The functor $X^\pm$ is representable by a $X^0$-affine algebraic space. \smallskip\\
iii) If $X/S$ is locally of finite presentation (resp. quasi-compact; resp. quasi-separated; resp. separated; resp. smooth; resp. is a scheme), so are $X^0$ and $X^\pm$.}\smallskip\\

Let $X/S$ be an algebraic space locally of finite presentation with an \'etale locally linearizable $\bbG_m$-action. There are maps locally of finite presentation of $S$-spaces
\begin{equation}\label{intro2}
\begin{tikzpicture}[baseline=(current  bounding  box.center)]
\matrix(a)[matrix of math nodes, 
row sep=1.0em, column sep=2em, 
text height=1.5ex, text depth=0.45ex] 
{ & X^{\pm}&\\ 
X^0&&X,\\ }; 
\path[->](a-1-2) edge node[above] {$q^{\pm}$\;}(a-2-1); 
\path[->](a-1-2) edge node[above] {\;\;\;\;$p^{\pm}$} (a-2-3); 
\end{tikzpicture}
\end{equation}
where $q^{\pm}$ is given by evaluating a morphism at the zero section, and $p^{\pm}$ by evaluating a morphism at the unit section. Let $n>1$ be a positive integer invertible on $S$, and denote by $D(X,\bbZ/n)$ the unbounded derived category of $(X_{\text{\'et}},\bbZ/n)$-modules, where $X_{\text{\'et}}$ is the \'etale topos of $X$. Let us define two functors from $D(X,\bbZ/n)$ to $D(X^0,\bbZ/n)$ by pull-push as follows
\begin{equation}\label{intro3}
\begin{aligned}
L_{X/S}^+&\defined  (q^+)_!\circ (p^+)^*\\
L_{X/S}^-&\defined  (q^-)_*\circ (p^-)^!.
\end{aligned}
\end{equation}
As in Braden's work \cite{Br03} (or Drinfeld-Gaitsgory's work \cite{DG15}) there exists a natural transformation of functors
\begin{equation}\label{intro4}
L_{X/S}^-\longto L_{X/S}^+.
\end{equation}
Let $a, p\co \bbG_{m,S}\times_S X\to X$ denote the action (resp. projection). We say a complex in $D(X,\bbZ/n)$ is (naively) $\bbG_m$-equivariant if there exists an isomorphism $a^*\calA\simeq p^*\calA$ in $D(\bbG_{m,S}\times_S X,\bbZ/n)$.  Let us define $D(X,\bbZ/n)^{\Gmon}$ to be the full subcategory strongly generated by $\bbG_m$-equivariant complexes, i.e. generated by a finite iteration of taking the cone of a morphism in $D(X,\bbZ/n)$. The objects in $D(X,\bbZ/n)^{\Gmon}$ are called $\bbG_m$-monodromic. \smallskip\\

\noindent {\bf Theorem B.} {\it Let $S$ be a scheme, and let $X/S$ be an algebraic space locally of finite presentation with an \'etale locally linearizable $\bbG_m$-action. Let $\calA\in D(X,\bbZ/n)^{\Gmon}$ be a bounded below complex. \smallskip\\
i) The arrow of $D(X^0,\bbZ/n)$
\[
L^-_{X/S}\calA\overset{\simeq}{\longto} L^+_{X/S}\calA
\]
is an isomorphism. In particular, the complex $L^-_{X/S}\calA$ is bounded below. \smallskip\\
ii) For any morphism of schemes $f\co S'\to S$, the isomorphism in i) is compatible with base change along $f_*$ and $f^*$. If $f$ is locally of finite type, it is also compatible with $f_!$ and $f^!$.}\smallskip\\

Let us point out the following consequence of Theorem B. Let $S$ be the spectrum of a henselian discrete valuation ring with generic point $\eta$ and special point $s$. Fix a geometric point $\bar{\eta}\to \eta$. Then there is the functor of nearby cycles
\begin{equation}\label{intro5}
\begin{aligned}
\Psi_X\co D(X_\eta,\bbZ/n) & \longto D(X_s\times_S\eta,\bbZ/n)\\
\calA & \longmapsto \bar{i}^*\bar{j}_*\calA_{\bar{\eta}},
\end{aligned}
\end{equation}
where $D(X_s\times_S\eta,\bbZ/n)$ is as in [SGA 7, XIII] the derived category of $((X_{\bar{s}})_{\text{\'et}},\bbZ/n)$-modules with a continuous action of the Galois group compatible with its action on $X_{\bar{s}}$. The usual functorialities of nearby cycles give transformations from $D(X_\eta,\bbZ/n)$ to $D(X^0_s\times_S\eta,\bbZ/n)$ as follows
\begin{equation}\label{intro6}
\begin{aligned}
L^-_{X_{\bar{s}}/\bar{s}}\circ\Psi_X &\longleftarrow \Psi_{X^0}\circ L^-_{X_\eta/\eta}, \\
L^+_{X_{\bar{s}}/\bar{s}}\circ \Psi_X&\longto \Psi_{X^0}\circ L^+_{X_\eta/\eta}.
\end{aligned}
\end{equation}


\noindent {\bf Corollary.} {\it Let $S$ be the spectrum of a henselian discrete valuation ring, and let $X/S$ be an algebraic space of finite type with an \'etale locally linearizable $\bbG_{m}$-action. Then, for $\calA\in D(X_\eta,\bbZ/n)$ bounded below, there is a commutative diagram in $D(X_s^0\times_S\eta,\bbZ/n)$ 
\begin{equation}\label{intro6}
\begin{tikzpicture}[baseline=(current  bounding  box.center)]
\matrix(a)[matrix of math nodes, 
row sep=1.5em, column sep=2em, 
text height=1.5ex, text depth=0.45ex] 
{ L^-_{X_{\bar{s}}/\bar{s}}\circ\Psi_X(\calA)& \Psi_{X^0}\circ L^-_{X_\eta/\eta}(\calA)\\ 
L^+_{X_{\bar{s}}/\bar{s}}\circ \Psi_X(\calA)&\Psi_{X^0}\circ L^+_{X_\eta/\eta}(\calA),\\ }; 
\path[->](a-1-2) edge (a-1-1); 
\path[->](a-1-1) edge  (a-2-1); 
\path[->](a-2-1) edge (a-2-2);
\path[->](a-1-2) edge  (a-2-2);
\end{tikzpicture}
\end{equation}
and all arrows are isomorphisms if $\calA$ is $\bbG_m$-monodromic.}\smallskip\\

\subsection{Link to the literature and strategy of proof} \label{secintro2} The commutativity of hyperbolic localization with nearby cycles is a purely formal consequence of Theorem B. Hence, sufficient generality is of importance: there are no finiteness assumptions imposed neither on the base scheme $S$ (e.g. locally noetherian) nor on the sheaves (e.g. constructible). Let us link the results to the literature\footnote{The author includes what he knows, but the outline is probably not complete. The author is grateful for every comment, e.g. if the reader feels that formulations are incorrect or other work should be mentioned. Of course, other comments or questions are equally welcome.}.\smallskip\\
{\bf Theorem A.} 
The spaces $X^\pm$ in \eqref{intro1} are defined by Drinfeld \cite{Dr13}, and he proves representability of $X^\pm$ for quasi-separated algebraic spaces of finite type over fields. Note that similar functors are studied by Hesselink \cite{He80}, where he proves representability under the existence of a $\bbG_m$-invariant affine open cover, i.e. the $\bbG_m$-action is Zariski locally linearizable. By results of Sumihiro \cite[Cor. 2]{Sum74}, \cite[Cor. 3.11]{Sum75} every $\bbG_m$-action on a normal variety has this property. Note that this fails without normality assumption, e.g. for $\bbP^1$ with $0$ and $\infty$ identified. More recently, Alper-Hall-Rydh \cite[Thm. 2.24]{AHR15} recover Drinfeld's result on $X^\pm$ by proving general results on the representability of $\Hom$-stacks. The condition of being \'etale locally linearizable comes from their generalization of Sumihiro's result \cite[Thm. 2.5]{AHR15}: every $\bbG_m$-action on a quasi-separated algebraic space locally of finite presentation is \'etale locally linearizable. In [\emph{loc. cit.}] this is shown for algebraically closed fields as bases and in upcoming work of Alper-Hall-Rydh \cite{AHR16} this hypothesis is removed. Theorem A is straight forward:\smallskip\\
\phantom{h}\hspace{1cm}(1) Prove Theorem A for $S$-affine schemes. \smallskip\\
\phantom{h}\hspace{1cm}(2) Descend the representability and favorable properties by using an \\ 
\phantom{h(2)}\hspace{1cm} equivariant atlas $\{U_i\to X\}$. \smallskip\\
Note that if $X$ is a scheme where the $\bbG_m$-action is not Zariski locally linearizable, then for the argument of Theorem A we have to leave the world of schemes, i.e. the fact that $X^\pm$ is a scheme in this case follows a posteriori from $X^\pm$ being an algebraic space. Note that our method is very close to Alper-Hall-Rydh's arguments \cite[\S 5.12]{AHR15}. We choose to include Theorem A because it makes the present manuscript self contained, and because we think it is of interest in its own: the hypothesis of being \'etale locally linearizable can be verified by hand in many cases. \smallskip\\
{\bf Theorem B.}  Braden \cite{Br03} proves that for a normal variety $X$ over an algebraically closed field, the transformation \eqref{intro4} is an isomorphism on weakly $\bbG_m$-equivariant complexes. Using Sumihiro's theorem, he reduces to the case of an affine space with a linear $\bbG_m$-action, and then uses a contraction argument \cite[Lem. 6]{Br03}. In \cite{DG15}, Drinfeld-Gaitsgory extend Braden's result to the case of quasi-separated algebraic spaces locally of finite type over characteristic zero fields in the context of $D$-modules. Their argument uses a certain family $\tilde{X}\to \bbA^1$, which is shown to be representable in \cite{Dr13}, and a sufficiently good six functor formalism, e.g. existence of a dualizing complex. Their method applies to the \'etale topology using $\bbQ_\ell$-sheaves with constructible cohomologies over fields of characteristic $\not= \ell$ \cite[\S 0.4]{DG15}. In our approach, we follow Braden's original method:\smallskip\\
\phantom{h}\hspace{1cm}(1) Prove that Theorem B i) holds for affine spaces with a linear $\bbG_m$-action. \smallskip\\
\phantom{h}\hspace{1cm}(2) Reduce to case (1) using an equivariant atlas $\{U_i\to X\}_i$: pull back to $U_i$ and \\
\phantom{h(2)}\hspace{1.1cm}embed $U_i$ into an affine space with a linear $\bbG_m$-action. \smallskip\\
A careful analysis of Braden's argument shows that in the presence of torsion coefficients no finiteness assumptions neither on $S$ nor on the sheaves are necessary. Theorem B ii) is proven by a diagram chase. The isomorphisms in Theorem B ii) are due to strong symmetry properties implied by Theorem B i), e.g. the functor $f_*$ commutes with $L^-$, and hence it also commutes with $L^+$ on monodromic complexes. 

The commutation of hyperbolic localization with nearby cycles for monodromic complexes follows from follows from Theorem B ii). M.~Finkelberg pointed out to us that a similar result in the complex analytic setting was proven earlier by Nakajima \cite[Prop 5.4.1 (2)]{N17}.

\subsection{Structure of the manuscript}
In \S 1, we study spaces with an \'etale locally linearizable $\bbG_m$-action, and prove Theorem A, cf. Theorem \ref{repthm} below. Paragraph \S 2 is devoted to the proof of Theorem B i), cf. Theorem \ref{Bradenthm} below. The toy case is $\bbA^1_S$, cf. \S \ref{moncomplexes}, and the argument for affine spaces with a linear $\bbG_m$-action in \S \ref{secconlem} and \S \ref{linearactions} builds upon it. In \S \ref{compclosedsec} and \S \ref{competalesec}, we deduce Theorem B i) from the latter case using a $\bbG_m$-equivariant atlas. The functorial properties in Theorem B ii) are studied in \S \ref{bcforBraden}, and the commutation of hyperbolic localization with nearby cycles is deduced in \S \ref{nbhccommute}.

\subsection{Acknowledgements} The author thanks T. Haines and J. Heinloth for their support and many useful comments on an earlier version. He warmly thanks J. Alper for explanations on the generalization of Sumihiro's theorem, and M.~Finkelberg for pointing out the reference \cite{N17}. He thanks A. Bouthier, M. Land, T. Nikolaus, M. Rapoport, P. Scholze and G. Williamson for discussions around the subject, and the referee for reviewing the manuscript. Further, he thanks F. Hamm and J. Lemessa for their interest and enthusiasm in mathematics. This work would not have been finished without them.  

\subsection{Notation}
For a scheme $S$, we denote by $\sch{S}$ the category of $S$-schemes. By a space $X/S$, we mean an algebraic space $X/S$ in the sense of \cite[Tag 025Y]{StaPro}: a sheaf on the big $\text{fppf}$-site
\[
X\co \sch{S}^{\text{op}}_{\text{fppf}}\longto \text{Set}
\]
with representable diagonal and which admits a surjective \'etale map from a scheme. In particular, we do not assume $X$ to be quasi-separated. Throughout we fix a general base scheme $S$. Special hypothesis on $S$ are spelled out explicitly when needed. For two sheaves $X$ and $Y$ on $\sch{S}_{\text{fppf}}$, we denote by $X\times Y=X\times_S Y$ the fiber product and by $\Hom_S(Y,X)$ the set of $S$-morphisms. 

\section{Spaces with \'etale locally linearizable $\bbG_m$-action}\label{attract} 
\subsection{General nonsense}
Let $S$ be a scheme, and let $X/S$ be a space. For an $S$-scheme $T$, let $X_T=X\times T$. For another space $Y/S$, define the contravariant set-valued functor $\underline{\Hom}_S(Y,X)$ on $\sch{S}_\text{fppf}$, for any $S$-scheme $T$, by
\[
\underline{\Hom}_S(Y,X)\co T\longmapsto \Hom_T(Y_T,X_T),
\]
Note that $\Hom_T(Y_T,X_T)=\Hom_S(Y_T,X)$. The functor $\underline{\Hom}_S(Y,X)$ is a sheaf on $\sch{S}_{\text{fppf}}$. For a morphism $f\co X'\to X$ of $S$-spaces, there is a transformation as follows
\begin{equation}\label{funchom1}
\underline{\Hom}_S(Y,X')\longto \underline{\Hom}_S(Y,X), \;\; x\longmapsto f\circ x.
\end{equation}
For a morphism $g\co Y'\to Y$ of $S$-spaces, there is a transformation as follows
\begin{equation}\label{funchom2}
\underline{\Hom}_S(Y,X)\longto \underline{\Hom}_S(Y',X), \;\; x\longmapsto x\circ g.
\end{equation}
For a morphism $S'\to S$ of schemes, there is an isomorphism as follows 
\begin{equation}\label{bchom}
\underline{\Hom}_{S'}(Y_{S'},X_{S'}) \overset{\simeq}{\longto} \underline{\Hom}_S(Y,X)\times S',
\end{equation}
which is compatible with \eqref{funchom1} and \eqref{funchom2}. 

Let $G/S$ be a $\text{fppf}$-sheaf of groups. If $X/S$  and $Y/S$ are equipped with a (left) $G$-action, then $G$ acts on $\underline{\Hom}_S(Y,X)$: for any $S$-scheme $T$ and $(g,x)\in G(T)\times \underline{\Hom}_S(Y,X)(T)$ define $g*x$ by 
\[
g*x\defined g\circ x \circ g^{-1},
\]
where $g$ (resp. $g^{-1}$) denotes the automorphism $X_T\to X_T$ (resp. $Y_T\to Y_T$) given by the $G$-action on $X$ (resp. $Y$). Define the subfunctor $\underline{\Hom}^{G}_S(Y,X)$ of $G$-equivariant morphisms from $Y$ to $X$, for any $S$-scheme $T$ by
\[
\underline{\Hom}^{G}_S(Y, X)\colon T\longmapsto \{x\in \Hom_T(Y_T,X_T)\;|\;\forall g\in G(T):\;\;g*x=x\}. 
\]
In other words, $\underline{\Hom}^{G}_S(Y, X)$ is the subfunctor of $G$-fixed points in $\underline{\Hom}_S(Y, X)$. 

\begin{lem}\label{basiclem} Let $G/S$ be a $\text{fppf}$-sheaf of groups. Let $X/S$ and $Y/S$ be spaces with $G$-action. \smallskip\\
i) The functor $\underline{\Hom}^{G}_S(Y, X)$ is a subsheaf of $\underline{\Hom}_S(Y, X)$ on $\sch{S}_{\text{fppf}}$. \\
ii) For a $G$-equivariant morphism $X'\to X$ (resp. $Y'\to Y$) of $S$-spaces, the transformation \eqref{funchom1} (resp. \eqref{funchom2}) restricts to a morphism on subsheaves
\[ 
\underline{\Hom}^G_S(Y,X')\longto \underline{\Hom}^G_S(Y,X)   \;\;\;\;\;\text{(resp. $\underline{\Hom}^G_S(Y,X)\longto \underline{\Hom}^G_S(Y',X)$)}.
\]
iii) For a morphism $S'\to S$ of schemes, the isomorphism \eqref{bchom} restricts to an isomorphism of subsheaves
\[
\underline{\Hom}^{G}_{S'}(Y_{S'},X_{S'}) \overset{\simeq}{\longto} \underline{\Hom}^G_S(Y,X)\times S',
\]
which is compatible with the transformations constructed in ii).
\end{lem}
\begin{proof}
For i), note that $\underline{\Hom}^{G}_S(Y, X)$ is the functor of fixed points with respect to the $G$-action on $\underline{\Hom}_S(Y, X)$, and hence a sheaf. Parts ii) and iii) are immediate.
\end{proof}

\begin{ex} \label{twochoices}
i) If $Y=S$, then $\underline{\Hom}^G_S(Y,X)=X^G$ is the functor of fixed points. If $X/S$ is a quasi-compact quasi-separated scheme and $G/S$ is a flat group scheme, then $X^G$ is representable by a closed subscheme of $X$, cf. Fogarty \cite{F73}. \\
ii) If $Y=G$ with the translation action, then $\underline{\Hom}^G_S(Y,X)=X$ by evaluating a morphism at the unit section of $G$.
\end{ex}

\subsection{Attractors, repellers and fixed points}
We are interested in the case where $G=\bbG_{m}$ is the multiplicative group viewed as a sheaf of groups on $\sch{S}_{\text{fppf}}$. Consider the following examples of schemes with $\bbG_m$-action:\smallskip\\
i) $Y=S$ equipped with the trivial $\bbG_{m}$-action;\smallskip\\
ii) $Y=(\bbA^1_S)^+$ where the underlying scheme is $\bbA^1_S$ equipped with the $\bbG_{m}$-action by dilations, i.e. for any $S$-scheme $T$ and $\la\in\bbG_{m,S}(T)=\calO_T^\times$, $x\in \bbA^1_S(T)=\calO_T$ the action is given by $(\la,x)\mapsto \la\cdot x$;\smallskip\\
iii) $Y=(\bbA^1_S)^-$ where the underlying scheme is $\bbA^1_S$ equipped with the opposite $\bbG_{m}$-action, i.e. the action is given by $(\la,x)\mapsto \la^{-1}\cdot x$.\smallskip\\

Drinfeld \cite{Dr13} introduces the following notations.

\begin{dfn}\label{Gmspaces}
Let $X/S$ be a space with $\bbG_{m}$-action. Define the sheaves $X^0$, $X^+$ and $X^-$ on $\sch{S}_{\text{fppf}}$ by
\begin{align*}
\hspace{1cm} X^0&\defined \underline{\Hom}^{\bbG_{m}}_S(S, X);\\
\hspace{1cm} X^+&\defined \underline{\Hom}^{\bbG_{m}}_S((\bbA_S^1)^+, X);\\
\hspace{1cm}X^-&\defined \underline{\Hom}^{\bbG_{m}}_S((\bbA^1_S)^-, X).
\end{align*}
The sheaf $X^0$ is called the \emph{space of fixed points}, $X^+$ the \emph{attractor} and $X^-$ the \emph{repeller}.
\end{dfn}

\begin{rmk}
i) The sheaf $X^0=X^{\bbG_{m}}$ is the functor of fixed points as in Remark \ref{twochoices}, i) above. In case ii) (resp. iii)), the sheaf $X^+$ (resp. $X^-$) is the functor of points \emph{floating to (resp. away from)} the fixed points. Informally speaking, the limit
\[
\underset{\la\to 0}{\lim}\;\la\cdot x \;\;\;\;\;\text{(resp.\;\;$\underset{\la\to \infty}{\lim}\;\la\cdot x$)}
\]
should exist.  \smallskip\\
ii) Note that $X^0$ and $X^\pm$ inherit $\bbG_m$-actions from $X$ (the trivial one on $X^0$), and with respect to these actions $(X^\pm)^0=X^0$. 
\end{rmk}

\begin{ex}\label{projspace} i) Let $X=\bbP^1_S$ with the natural $\bbG_{m}$-action. Then $X^0=\{0_S\}\amalg\{\infty_S\}$ are the fixed points, $X^+=(\bbA_S^1)^+\amalg \{\infty_S\}$ and $X^-=(\bbA^1_S)^-\amalg\{0_S\}$. In particular, $X^+$ and $X^-$ are representable. \smallskip \\
ii) Let $X=G$ be a $S$-group scheme with $\bbG_{m}$-action given by conjugation with a cocharacter $\la\co \bbG_{m,S}\to G$. Then $X^0=Z_G(\la)$ is the centralizer of $\la$ and $X^\pm=P(\pm\la)$ are the `\emph{parabolic}' subgroups defined by the dynamic method, cf. \cite[Thm. 4.1.7]{Co14}. 
\end{ex}

\begin{dfn}
Let $S$ be a scheme, and let $X/S$ be a space. A $\bbG_{m}$-action on $X$ is called \emph{\'etale locally linearizable} if there exists a $\bbG_{m}$-equivariant covering family
\begin{equation}\label{cover}
\{U_i\longto X\}_i,
\end{equation}
where $U_i$ are $S$-affine schemes with $\bbG_m$-action and the maps $U_i\to X$ are \'etale.
\end{dfn} 

\begin{rmk}\label{loclin}
i) That the family $\{U_i\longto X\}_i$ is covering means that the map $\coprod_iU_i\to X$ is surjective on the underlying topological spaces. \smallskip\\
ii) The attribute `\emph{linearizable}' refers to the fact that an affine scheme of finite presentation with $\bbG_m$-action can be (Zariski locally on the base) equivariantly embedded as a closed subscheme into some affine space on which $\bbG_m$-acts linearly, cf. Lemma \ref{embedlem} below. \smallskip\\
iii) If $S$ is the spectrum of an algebraically closed field, then every $\bbG_m$-action on a quasi-separated algebraic space $X/S$ locally of finite presentation is \'etale locally linearizable, cf. \cite[Thm. 2.5]{AHR15}. In forthcoming work of Alper-Hall-Rydh \cite{AHR16} \'etale locally linearizability is shown for an arbitrary base scheme $S$.  
\end{rmk}

\begin{thm}\label{repthm} Let $S$ be a scheme, and let $X/S$ be a space with an \'etale locally linearizable $\bbG_{m}$-action. Let $\{U_i\to X\}_i$ be a $S$-affine $\bbG_m$-equivariant \'etale covering family.\smallskip\\
i) The subfunctor $X^0$ of $X$ is representable by a closed subspace, and the induced family $\{U^0_i\to X^0\}_i$ is $S$-affine, \'etale and covering.\smallskip\\
ii) The functors $X^\pm$ are representable by algebraic spaces, and the induced family $\{U^\pm_i\to X^\pm\}_i$ is $S$-affine, \'etale, $\bbG_m$-equivariant and covering.\smallskip\\
iii) If $X/S$ is locally of finite presentation (resp. quasi-compact; resp. quasi-separated; resp. separated; resp. smooth; resp. is a scheme), so are $X^0$ and $X^\pm$.
\end{thm}

This theorem, combined with Corollary \ref{attractorcor} below, implies Theorem A from the introduction. The proof of part i) is in \S\ref{spaceoffixedpoints} and of part ii) and iii) in  \S\ref{attractorandrepeller} below. The strategy is to descend the desired properties from the equivariant atlas. The keys are Lemmas \ref{bcforzero} and \ref{bcforplus}, cf. also \cite[Lem. 5.8]{AHR15} over fields. As it turns out $X^\pm$ is an affine $X^0$-space (Corollary \ref{attractorcor}) which implies part iii). Let us warm up with an easy case. 

\subsection{The affine case} \label{theaffinecase}
Let $X=\Spec_S(\calB)$ be an $S$-affine scheme where $\calB$ denotes a quasi-coherent $\calO_S$-algebra. If $S$ is connected, a $\bbG_{m}$-action on $X/S$ is equivalent to a $\bbZ$-grading 
\begin{equation}\label{grading}
\calB=\oplus_{i\in \bbZ}\calB_i
\end{equation}
on the $\calO_S$-algebra $\calB$, i.e. \eqref{grading} as quasi-coherent $\calO_S$-modules and $\calB_i\cdot \calB_j\subset \calB_{i+j}$. Let $\calI^+$ (resp. $\calI^-$, resp. $\calI^0$) be the quasi-coherent\footnote{The sheaf $\calI^+$ is quasi-coherent because it is the image of a quasi-coherent sheaf: $\calB\times (\oplus_{i<0}\calB_i)\to \calB$, $(b,c)\mapsto b\cdot c$.} ideal sheaf in $\calB$ generated by the homogeneous elements of strictly negative (resp. strictly positive, resp. non-zero) degree. 

\begin{lem}\label{affinecase}
i) The functor $X^0$ is representable by the closed subscheme of $X$ defined by $\calI^0$. \\
ii) The functor $X^\pm$ is representable by the closed subscheme of $X$ defined by $\calI^\pm$.   
\end{lem}  
\begin{proof}
Let $p\co T\to S$ be a scheme. Since $X$ is $S$-affine the set $X(T)$ identifies with set of $\calO_S$-algebra morphisms $\calB\to p_*\calO_T$, and $X^0(T)$ is the subset of $\bbZ$-graded $\calO_S$-algebra morphisms $\calB\to p_*\calO_T$, where $p_*\calO_T$ is in degree $0$. This implies part i). Likewise, the set $X^+(T)$ (resp. $X^-(T)$) identifies with the set of $\bbZ$-graded $\calO_S$-algebra morphisms $\calB\to p_*\calO_T[t]$ where the parameter $t$ has degree $1$ (resp. $-1$). This implies part ii). 
\end{proof}

\subsection{The space of fixed points $X^0$} \label{spaceoffixedpoints}
It is suprising that $X^0\subset X$ is closed, even if $X$ is not separated. This is closely related to the connectedness of $\bbG_m$. Let us prepare for the proof.

\begin{lem}\label{bcforzero}
Let $U\to X$ be $\bbG_m$-equivariant \'etale $S$-morphism. Then as functors
\[
U^0= U\times_XX^0.
\]
\end{lem}

\begin{proof} Let $T/S$ be a scheme. An element $\varphi\in (U\times_XX^0)(T)$ corresponds to a commutative diagram
\[
\begin{tikzpicture}[baseline=(current  bounding  box.center)]
\matrix(a)[matrix of math nodes, 
row sep=1.5em, column sep=2em, 
text height=1.5ex, text depth=0.45ex] 
{&U \\ 
T&X \\}; 
\path[->](a-2-1) edge node[above] {$\tilde{f}$} (a-1-2);
\path[->](a-1-2) edge node[right] {\text{\'et}} (a-2-2);
\path[->](a-2-1) edge node[below] {$f$} (a-2-2);
\end{tikzpicture}
\]
where $f$ is $\bbG_m$-equivariant. We have to show that $\tilde{f}$ is $\bbG_m$-equivariant. It is enough to show equivariance \'etale locally. If $T$ is the spectrum of a strictly henselian local ring, then $\tilde{f}$ is 
$\bbG_{m}$-equivariant (because $\bbG_m$ is connected and $U\to X$ is \'etale). Since $U\to X$ is locally of finite presentation, we get
\[
\Hom_X^{\bbG_m}(\lim_iT_i,U)=\colim_i\Hom_X^{\bbG_m}(T_i,U)
\]
for any cofiltered limit $\lim_iT_i$ of affine schemes in $\sch{X}_{\text{fppf}}$. Hence, $\tilde{f}$ is \'etale locally $\bbG_m$-equivariant. The lemma follows. 
\end{proof}

\begin{proof}[Proof of Theorem \ref{repthm} i)]
Let us show that $X^0$ is representable by a closed subspace, and $\{U_i^0\to X^0\}_i$ is $S$-affine, \'etale and covering. By Lemma \ref{affinecase} i), the $U_i^0$ are closed subschemes of $U_i$, hence $S$-affine. Lemma \ref{bcforzero} shows that the following commutative diagram
\begin{equation}\label{bcrepzero}
\begin{tikzpicture}[baseline=(current  bounding  box.center)]
\matrix(a)[matrix of math nodes, 
row sep=1.5em, column sep=2em, 
text height=1.5ex, text depth=0.45ex] 
{\coprod_iU_i^0&X^0 \\ 
\coprod_iU_i&X \\}; 
\path[->](a-1-1) edge (a-1-2);
\path[->](a-1-1) edge (a-2-1);
\path[->](a-1-2) edge (a-2-2);
\path[->](a-2-1) edge (a-2-2);
\end{tikzpicture}
\end{equation}
is cartesian. Now the representability of $X^0$ follows from \cite[Tag 03I2]{StaPro} applied to the transformation $X^0\to X$ and the property `\emph{closed immersion}': (i) being a closed immersion is stable under base change, $\text{fppf}$-local on the base and closed immersions satisfy $\text{fppf}$-descent (because they are affine); (ii) $X^0$ is a sheaf; (iii) $X$ is an algebraic space; (iv) the bottom arrow in \eqref{bcrepzero} is surjective and \'etale, and $\coprod_iU_i^0$ is representable; (v) the left vertical arrow in \eqref{bcrepzero} is a closed immersion. This implies that $X^0$ is an algebraic space and $X^0\to X$ is a closed immersion. This proves Theorem \ref{repthm} i).
\end{proof}

\subsection{Attractors and repellers $X^\pm$}\label{attractorandrepeller}
For an \'etale locally linearizable $\bbG_{m}$-action the representability of $X^\pm$ is proven similarly. Let us explain the argument.\\
Note that under the morphism $\bbG_m\to \bbG_m$, $\la\mapsto \la^{-1}$ the notions of $X^+$ and $X^-$ are interchanged. Hence, it is enough to prove representability of $X^+$. Let us denote $(\bbA^1_S)^+$ by $\bbA^1_S$ in this subsection.\\
The zero section $S\to \bbA^1_S$ is $\bbG_{m}$-equivariant and defines by functoriality a morphism $X^+\to X^0$.

\begin{lem} \label{bcforplus}
Let $U\to X$ be a $\bbG_{m}$-equivariant \'etale $S$-morphism where $U$ is an $S$-affine scheme. Then as functors
\[
U^+=U^0\times_{X^0}X^+.
\]
\end{lem}

\begin{proof}
Let $p\co T\to S$ be a scheme, and let $\varphi\in (U^0\times_{X^0}X^+)(T)$. The element $\varphi$ corresponds to a commutative diagram of $\bbG_{m}$-equivariant morphisms
\[
\begin{tikzpicture}[baseline=(current  bounding  box.center)]
\matrix(a)[matrix of math nodes, 
row sep=1.5em, column sep=2em, 
text height=1.5ex, text depth=0.45ex] 
{T&U \\ 
\bbA^1_T&X \\}; 
\path[->](a-1-1) edge (a-1-2);
\path[->](a-1-1) edge (a-2-1);
\path[->](a-1-2) edge (a-2-2);
\path[->](a-2-1) edge (a-2-2);
\end{tikzpicture}
\] 
where $T\to \bbA^1_T$ is the zero section. Let us construct a unique $\bbG_m$-equivariant lift $\bbA^1_T\to U$ over $X$. For $i\geq 0$ denote $\bbA^1_{T,i}=\Spec_T(\calO_T[t]/t^{i+1})$ the $i$-th infinitesimal neighbourhood of the zero section $T\to \bbA^1_T$. Since $U\to X$ is \'etale, there is a unique $\bbG_m$-equivariant lift $\colim_i\bbA^1_{T,i}\to U$ over $X$. We claim that
\begin{equation}\label{formalnbh}
\Hom_S^{\bbG_m}(\colim_i\bbA^1_{T,i},U)=\Hom_S^{\bbG_m}(\bbA^1_T,U).
\end{equation}
Let $U=\Spec_S(\calR)$ for a quasi-coherent $\calO_S$-algebra $\calR$. Assume that $S$ is connected, and let $\calR=\oplus_{i\in \bbZ}\calR_i$ be the grading corresponding to the $\bbG_m$-action. An element of the left hand side of \eqref{formalnbh} corresponds to a morphism of $\calO_S$-algebras
\begin{equation} \label{pothom}
\calR\longto p_*\calO_T\pot{t}
\end{equation}
compatible with $\bbZ$-gradings (the parameter $t$ has degree $1$). As $\calR=\oplus_{i\in \bbZ}\calR_i$ each morphism \eqref{pothom} factors through the subalgebra $\oplus_{i\geq 0}p_*\calO_T\cdot t^i=p_*\calO_T[t]$ of $p_*\calO_T\pot{t}$, i.e. defines a $\bbZ$-graded $\calO_S$-algebra morphism $\calR\to p_*\calO_T[t]$. This proves \eqref{formalnbh} and implies the lemma.
\end{proof}

\begin{proof}[Proof of Theorem \ref{repthm} ii).]
We show that $X^+$ is an algebraic space, and that the family $\{U_i^+\to X^+\}$ is $S$-affine, \'etale and covering. Note that $\bbG_m$-equivariance follows from functoriality. By Lemma \ref{affinecase}, the sheaf $U_i^+$ is representable by a closed subscheme of $U_i$. Lemma \ref{bcforplus} shows that the following commutative diagram of sheaves
\begin{equation}\label{bcrep}
\begin{tikzpicture}[baseline=(current  bounding  box.center)]
\matrix(a)[matrix of math nodes, 
row sep=1.5em, column sep=2em, 
text height=1.5ex, text depth=0.45ex] 
{\coprod_iU_i^+&X^+ \\ 
\coprod_iU_i^0&X^0 \\}; 
\path[->](a-1-1) edge (a-1-2);
\path[->](a-1-1) edge (a-2-1);
\path[->](a-1-2) edge (a-2-2);
\path[->](a-2-1) edge (a-2-2);
\end{tikzpicture}
\end{equation}
is cartesian where $X^+\to X^0$ (resp. $U_i^+\to U_i^0$) is induced by the zero section $S\to \bbA^1_S$. By Theorem \ref{repthm} i), the family $\{U_i^0\to X^0\}_i$ is \'etale and covering. Now the representability of $X^+$ follows from \cite[TAG 03I2]{StaPro} applied to the transformation $X^+\to X^0$ and the property `\emph{affine}': (i) being affine is stable under base change, $\text{fppf}$-local on the base and affine morphisms satisfy $\text{fppf}$-descent; (ii) $X^+$ is a sheaf; (iii) $X^0$ is an algebraic space by Theorem \ref{repthm} i); (iv) the bottom arrow in \eqref{bcrep} is surjective and \'etale, and $\coprod_iU_i^+$ is representable; (v) the left vertical arrow in \eqref{bcrep} is affine. This implies that $X^+$ is an algebraic space and $X^+\to X^0$ is affine. This proves Theorem \ref{repthm} ii). 
\end{proof}

\begin{cor} \label{attractorcor}
The map $X^+\to X^0$ induced by the zero section $S\to \bbA^1_S$ is affine, has geometrically connected fibers and induces a bijection on the sets of connected components $\pi_0(|X^+|)\simeq \pi_0(|X^0|)$ of the underlying topological spaces.
\end{cor}
\begin{proof} Affineness of $X^+\to X^0$ is proven above, and we show that the fibers are connected. Let $K$ be a field, and let $x\co \Spec(K)\to X^0$ be a point. Denote $X^+_x=X^+\times_{X^0, x}\Spec(K)$. We claim that its underlying topological space $|X^+_x|$ is connected, cf. \cite[Algebraic Spaces, \S 4]{StaPro} for the definition of underlying topological spaces. Let $y\co \Spec(L)\to X^+_x$ be a point, and denote by $x_L$ the composition $\Spec(L)\to \Spec(K)\overset{x}{\to} X^0$. Then $x_L$ and $x$ define the same point of $|X^+_x|$. The key observation is that the $\bbG_m$-action on $X^+$ extends to an action of the multiplicative monoid $\bbA^1$ in the obvious way. Hence, the $\bbA^1$-orbit of $y$ defines a map $h\co \bbA^1_L\to X^+_x$ with $h(1)=y$ and $h(0)=x_L$ because $(X^+_x)^0=x$, cf. Lemma \ref{affinecase} i). Thus, the points defined by $y$ and $x$ lie in the connected set $|h|(|\bbA^1_L|)$. Since $y$ was arbitrary, this shows that $|X^+_x|$ is connected. In particular, the continuous map $|X^+|\to |X^0|$ has connected fibers, and the assertion on connected components follows from the existence of a section $|X^0|\subset |X^+|$. This shows the corollary.  
\end{proof}

\begin{proof}[Proof of Theorem \ref{repthm} iii).]
Let us check the list of properties. If $X$ is locally of finite presentation, then it is immediate from the definition that $X^0$ and $X^+$ are locally of finite presentation, i.e. the functors commute with cofiltered limits of affine schemes in $\sch{S}_{\text{fppf}}$. Let $\calP$ be one of the following properties: quasi-compact, quasi-separated, separated, being a scheme. If $X$ has property $\calP$, so has $X^0$ (because $X^0\subset X$ is closed). Since $X^+\to X^0$ is affine, in particular representable, quasi-compact and (quasi-)separated, and each property is stable under composition, it follows that $X^+$ has property $\calP$, if $X^0$ has. Now let $X/S$ be smooth. We claim that the maps $X^0\to S$ and $X^+\to X^0$ are smooth (hence is $X^+\to S$). As smoothness may be checked \'etale locally, diagram \eqref{bcrep} reduces us to the case where $X/S$ is affine. Then the claim is the main result of Margaux \cite[Thm 1.1]{Mar15}, cf. also Remark 1.2 and Remark 3.3. of [\emph{loc.\! cit.}]. This finishes the proof of Theorem \ref{repthm} iii). 
\end{proof}

\subsection{The hyperbolic localization diagram}
Let us relate the spaces $X^0$ and $X^{\pm}$ to each other. The structure morphism $(\bbA^1_S)^{\pm}\to S$ is $\bbG_{m}$-equivariant which defines by functoriality, cf. Lemma \ref{basiclem} ii), a transformation
\[
i^{\pm}\co X^0\longto X^{\pm}.  
\] 
Further, the zero section $S\to (\bbA^1_S)^{\pm}$ is $\bbG_{m}$-equivariant which defines, again by functoriality, a transformation\footnote{Informally speaking, a point $x\in X^\pm$ maps to its limit $\lim_{\la\to 0}\la\cdot x$ (resp. $\lim_{\la\to\infty}\la\cdot x$).}
\[
q^{\pm}\co X^{\pm}\longto X^{0},
\]
such that $q^\pm\circ i^\pm=\id$. Likewise, the inclusion $\bbG_{m,S}\to (\bbA^1_S)^{\pm}$ is $\bbG_{m}$-equivariant which defines, by Lemma \ref{basiclem} ii) and Example \ref{twochoices} ii), a transformation
\[
p^{\pm}\co X^{\pm}\longto X,
\]
such that $p^+\circ i^+=p^{-}\circ i^{-}$ is the inclusion of the subfunctor $X^0\subset X$.

\begin{ex}
Let $X=\bbP^1_S$ as in Example \ref{projspace}. Then $i^{\pm}\co \{0_S\}\amalg \{\infty_S\}\to X^{\pm}$ is the inclusion and in particular closed. The morphism $q^{\pm}:X^{\pm}\to \{0_S\}\amalg \{\infty_S\}$ is given by contracting the $\bbA^1$-components of $X^{\pm}$. The morphism $p^{\pm}\co X^{\pm}\to \bbP^1_S$ is the inclusion and in particular a monomorphism (but not locally closed). 
\end{ex}

\begin{dfn} \label{hyperbolicloc} Let $X/S$ be a space with a $\bbG_{m}$-action. The commutative diagram $\HypLoc{X}$ of transformations
 \[
\begin{tikzpicture}[baseline=(current  bounding  box.center)]
\matrix(a)[matrix of math nodes, 
row sep=2em, column sep=2em, 
text height=1.5ex, text depth=0.45ex] 
{X^0&&&\\
&X^+\times_{X}X^-&X^-& X^0\\ 
 &X^+&X &  \\
 &X^0,& & \\}; 
\path[->](a-1-1) edge node[above] {$j$} (a-2-2);  
\path[->](a-2-2) edge node[below] {$'p^-$} (a-2-3); 
\path[->](a-2-2) edge node[right] {$'p^+$} (a-3-2); 
\path[->](a-1-1) edge[bend left=20] node[above] {$i^-$} (a-2-3); 
\path[->](a-1-1) edge[bend right=20] node[left] {$i^+$} (a-3-2); 
\path[->](a-3-2) edge node[below] {$p^+$} (a-3-3);
\path[->](a-2-3) edge node[above] {$q^-$} (a-2-4);
\path[->](a-3-2) edge node[left] {$q^+$} (a-4-2);
\path[->](a-2-3) edge node[right] {$p^-$} (a-3-3);
\end{tikzpicture}
\]
is called the hyperbolic localization diagram.
\end{dfn}

\begin{rmk} 
In view of the explicit description in Lemma \ref{affinecase}, the map $j$ is an isomorphism if $X$ is $S$-affine. In general, $X^+\times_{X}X^-$ is strictly bigger, e.g. for $X=\bbP^1_S$. See Proposition \ref{morprop} iii) below for the basic property of the map.
\end{rmk}

In view of Theorem \ref{repthm} and the definitions, we obtain the following corollary.

\begin{cor}\label{bccor}
Let $S$ be a scheme, and let $X/S$ be a space with an \'etale locally linearizable $\bbG_{m}$-action. For a morphism of schemes $S'\to S$, the induced $\bbG_m$-action on $X'=X_{S'}$ is again \'etale locally linearizable. The transformation constructed in Lemma \ref{basiclem} iii) defines a $\bbG_{m}$-equivariant isomorphism of commutative diagrams of $S'$-spaces
\[
\HypLoc{X'}\overset{\simeq}{\longto} \HypLoc{X}\times S'.
\]
\end{cor}
\hfill\ensuremath{\Box} 

Let us mention some basic properties of the morphisms appearing in $\HypLoc{X}$.

\begin{prop} \label{morprop} Let $S$ be a scheme, and let $X/S$ be a space with an \'etale locally linearizable $\bbG_{m}$-action. The morphism of $S$-spaces\smallskip\\
i) $i^\pm\co X^0\to X^\pm$ is a closed immersion;\smallskip\\
ii) $q^{\pm}\co X^\pm\to X^0$ is affine, has geometrically connected fibers and induces a bijection on connected components;\smallskip\\
iii) $j=(i^+,i^-)\co X^0\to X^+\times_XX^-$ is an open and closed immersion. 
\end{prop} 

\begin{rmk} Part iii) is \cite[Prop. 1.6.2]{Dr13} over fields.
\end{rmk}

\begin{proof} Part i) follows from the fact that $p^+\circ i^+=p^{-}\circ i^{-}$ is the inclusion of the closed subspace $X^0\subset X$. Part ii) is Corollary \ref{attractorcor}. Consider part iii). Clearly, $j$ is a closed immersion, and we show that it is also \'etale. Let $\{U_i\to X\}_i$ be a $S$-affine $\bbG_m$-equivariant \'etale covering family. Lemma \ref{bcforzero} shows that the commutative diagram of $S$-spaces
\begin{equation}\label{bcopen}
\begin{tikzpicture}[baseline=(current  bounding  box.center)]
\matrix(a)[matrix of math nodes, 
row sep=1.5em, column sep=2em, 
text height=1.5ex, text depth=0.45ex] 
{\coprod_iU_i^0&X^0 \\ 
\coprod_{i,j}(U_i^+\times_{X}U_j^-)&X^+\times_XX^- \\}; 
\path[->](a-1-1) edge (a-1-2);
\path[->](a-1-1) edge (a-2-1);
\path[->](a-1-2) edge (a-2-2);
\path[->](a-2-1) edge (a-2-2);
\end{tikzpicture}
\end{equation}
is cartesian (apply the lemma to $U_i^\pm \to X^\pm$ and use $(X^\pm)^0=X^0$). The bottom arrow in \eqref{bcopen} is \'etale and surjective. By descent it is enough to show that $U^0_i\to U_i^+\times_XU_i^-$ is \'etale. Since $U_i\to X$ is \'etale, the diagonal $U_i\to U_i\times_XU_i$ is \'etale, and hence the morphism $U_i^+\times_{U_i}U_i^-\to U_i^+\times_XU_i^-$ obtained by base change is \'etale. Now the explicit description in Lemma \ref{affinecase} shows that $U_i^0=U_i^+\times_{U_i}U_i^-$. The proposition follows.
\end{proof}

\begin{rmk} i) If $X/S$ is separated, then $p^\pm\co X^{\pm}\to X$ is a monomorphism. In general this fails, e.g. for the affine line with double origin. \smallskip\\
ii) It is suprising that $j$ is an open immersion even for non-normal schemes, e.g. if $X$ is $\bbP^1$ with $0$ and $\infty$ identified. Then $X^0=\{*\}$ but $X^\pm\not= X$ as one might guess. Indeed, consider the $\bbG_{m}$-equivariant projection $\bbP^1\to X$. It induces an isomorphism $(\bbA^1)^\pm\subset (\bbP^1)^\pm\to X^\pm$, and one gets $X^+\times_XX^-=\bbG_{m}\amalg \{*\}$. 
\end{rmk}

\section{Hyperbolic localization in families}\label{hlfsec}
Let $S$ be a scheme, and let $X/S$ be a space. Let $\La=\bbZ/n$ with $n>1$ invertible on $S$. We denote by $D(X,\La)$ the unbounded derived category of $(X_{\text{\'et}},\La)$-modules, where $X_{\text{\'et}}$ denotes the \'etale topos associated with $X$. If $f\co Y\to X$ is a morphism of $S$-spaces, there are the Grothendieck operations
\begin{align*}
f^*\co D(X,\La)&\to D(Y,\La),\;\;\;\; f_*\co D(Y,\La)\to D(X,\La),\\
\str\otimes_X\str\co &D(X,\La)\times D(X,\La)\longto D(X,\La),\\
\Hom_X\co &D(X,\La)^{\text{op}}\times D(X,\La)\longto D(X,\La),
\end{align*}
where $(f^*,f_*)$ and $(\str\otimes_X\calA,\Hom_X(\calA,\str))$ for every $\calA\in D(X,\La)$ are pairs of adjoint functors. If $f\co Y\to X$ is locally of finite type, there is another pair $(f_!,f^!)$ of adjoint functors
\[
f_!\co D(Y,\La)\to D(X,\La),\;\;\;\; f^!\co D(X,\La)\to D(Y,\La).
\]
These operations satisfy the usual properties: base change, projection and K\"unneth formula, trace map and Poincar\'e duality, cf. the work of Liu-Zheng \cite{LZ12} for all statements in full generality and the references cited there. \\
Further, if $Z\overset{i}{\to} X\overset{j}{\leftarrow} X\bslash Z$ is a closed immersion with open complement, then for any $\calA\in D(X,\La)$ there are distinguished triangles 
\begin{align*}
j_!j^*\calA&\longto \calA\longto i_*i^*\calA\longto   \\
i_*i^!\calA & \longto \calA \longto j_*j^*\calA\longto,
\end{align*}
and we have $i^*j_!=0=i^!j_*$, $i^*i_*\simeq \id\simeq i^!i_*$ and $j^*j_!\simeq \id\simeq j^*j_*$, cf. \cite[TAG 0A4L]{StaPro} for the second triangle.


\subsection{Construction of $L^-_{X/S}\to L^+_{X/S}$}
Let $X/S$ be a space locally of finite presentation with an \'etale locally linearizable $\bbG_{m}$-action. By Theorem \ref{repthm} the hyperbolic localization diagram $\HypLoc{X}$ in \eqref{hyperbolicloc} is a commutative diagram of $S$-spaces locally of finite presentation. Let us define two functors by pull-push along the maps
\[
\begin{tikzpicture}[baseline=(current  bounding  box.center)]
\matrix(a)[matrix of math nodes, 
row sep=1.3em, column sep=2em, 
text height=1.5ex, text depth=0.45ex] 
{ & X^{\pm}&\\ 
X^0&&X.\\ }; 
\path[->](a-1-2) edge node[above] {$q^{\pm}$\;}(a-2-1); 
\path[->](a-1-2) edge node[above] {\;\;\;\;$p^{\pm}$} (a-2-3); 
\end{tikzpicture}
\]

\begin{dfn} \label{hyplocfun} The functors $L_{X/S}^{\pm}\co D(X,\La)\to D(X^0,\La)$ are defined as
\begin{align*}
L_{X/S}^+\defined (q^+)_{!}\circ (p^+)^*\\
L_{X/S}^-\defined (q^-)_*\circ (p^-)^!.\\
\end{align*}
\end{dfn}

\begin{construction} Braden \cite{Br03} (cf. also \cite{DG15}) constructs a transformation of functors 
\begin{equation}\label{locmorph}
L_{X/S}^-\longto L_{X/S}^+
\end{equation}
as follows. By Proposition \ref{morprop} iii), the morphism 
\begin{equation}\label{hyperdis}
j\defined (i^+,i^-)\co X^0 \longto X^+\times_X X^-
\end{equation}
is an open and closed immersion, hence $(j^!,j_*)$ are adjoint. Recall the notation of the maps in the definition of \HypLoc{X}, cf. Definition \ref{hyperbolicloc}. Applying the functor $(i^-)^* (p^-)^!$ to the unit of the adjunction $\id\to (p^+)_{*}(p^+)^*$, we get 
\begin{equation}\label{compD1}
\begin{aligned}
(i^-)^{*} (p^-)^!\longto\;\; &(i^-)^* (p^-)^!(p^+)_{*}(p^+)^*\\
\simeq\;\; &(i^-)^* ('p^-)_{ *} ('p^+)^! (p^+)^* && \;\;\;\;\;(\text{base change})\\
\longto\;\; &(i^-)^* ('p^-)_{ *} j_*j^! ('p^+)^! (p^+)^* &&\;\;\;\;\;((j^!,j_*)\text{-adjunction})\\
\simeq\;\; &(i^+)^! (p^+)^*. && \;\;\;\;\;(('p^\pm)\circ j =i^\pm\;\text{and}\; (i^-)^*(i^-)_*\simeq \id)
\end{aligned}
\end{equation}
Now precompose (resp. compose) \eqref{compD1} with the transformation
\begin{equation}\label{contrafos} 
(q^-)_*\longto (i^-)^* \;\;\;\;\text{(resp. $(i^+)^!\longto (q^+)_!$.)}
\end{equation}
obtained from the unit $\id\to (i^-)_*(i^-)^*$ (resp. the counit $(i^+)_{!}(i^+)^!\to \id$) by applying $(q^-)_*$ (resp. $(q^+)_{!}$) and using $q^\pm\circ i^\pm=\id$. This constructs \eqref{locmorph}.
\end{construction}

\subsection{Monodromic complexes} Let $a,p\co \bbG_{m,S}\times X\to X$ be the action (resp. projection) map. We call a complex $\calA\in D(X,\La)$ \emph{(naively) $\bbG_m$-equivariant} if there exists an isomorphism $a^*\calA\simeq p^*\calA$ in $D(\bbG_{m,S}\times X,\La)$.

\begin{dfn} \label{mondfn}
Let $D(X,\La)^{\Gmon}$ be the full subcategory of $D(X,\La)$ strongly generated\footnote{Strongly generated means generated by a finite iteration of taking the cone of a morphism in $D(X,\La)$.} by $\bbG_m$-equivariant complexes. The objects of $D(X,\La)^{\Gmon}$ are called \emph{$\bbG_m$-monodromic complexes}.
\end{dfn}

\begin{rmk} Note that Definition \ref{mondfn} differs from Drinfeld-Gaitsgory's definition of $\bbG_m$-mono-dromic complexes. Let us recall their definition. The quotient $\calX=[\bbG_m\bslash X]$ is an Artin stack over $S$, and we denote by $D(\calX,\La)$ the unbounded derived category of $(\calX_{\lisset},\La)$-modules, where $\calX_\lisset$-denotes the lisse-\'etale topos associated with $\calX$. Let $D_{\text{cart}}(\calX,\La)$ be the full subcategory of $D(\calX,\La)$ spanned by complexes whose cohomology sheaves are cartesian. The canonical projection $\pi\co X\to \calX$ is smooth, and hence we get a morphism of topoi $X_{\lisset}\to \calX_{\lisset}$. Note that $D_{\text{cart}}(X_{\lisset}, \La)=D(X_{\text{\'et}},\La)$. At the level of derived categories, this gives the pullback functor 
\[
\pi^*\colon D_{\text{cart}}(\calX,\La) \to D(X,\La).
\]
Drinfeld-Gaitsgory define the category of $\bbG_m$-monodromic complexes to be the full subcategory strongly generated by the essential image of $\pi^*$. Note that each complex of the form $\pi^*\calA$ admits a natural isomorphism
\[
a^*(\pi^*\calA)\overset{\simeq}{\longto} p^*(\pi^*\calA)
\]
in $D(\bbG_{m,S}\times X,\La)$. Hence, each $\bbG_m$-monodromic complex in the sense of Drinfeld-Gaitsgory is $\bbG_m$-monodromic in the sense of Definition \ref{mondfn}. However, there may exist (naively) $\bbG_m$-equivariant complexes that are not in the essential image of $\pi^*$.
\end{rmk}

\begin{lem}\label{funcmonlem}
Let $f:Y\to X$ be a $\bbG_{m}$-equivariant morphism of $S$-spaces. \smallskip\\
i) The pair $(f^*, f_*)$  restricts to a pair of adjoint functors on $\bbG_m$-monodromic complexes.\smallskip\\
ii) If $f$ is locally of finite type, the pair $(f_!,f^!)$ restricts to a pair of adjoint functors on $\bbG_m$-monodromic complexes.
\end{lem}
\begin{proof}
By $\bbG_{m}$-equivariance we get a cartesian diagram of $S$-spaces
\[
\begin{tikzpicture}[baseline=(current  bounding  box.center)]
\matrix(a)[matrix of math nodes, 
row sep=1.5em, column sep=2em, 
text height=1.5ex, text depth=0.45ex] 
{\bbG_{m,S}\times Y& \bbG_{m,S}\times X \\ 
Y & X \\}; 
\path[->](a-1-1) edge node[above] {$f$} (a-1-2);
\path[->](a-2-1) edge node[above] {$f$} (a-2-2);
\path[->](a-1-1) edge node[left] {$a$} (a-2-1);
\path[->](a-1-2) edge node[right] {$a$} (a-2-2);
\end{tikzpicture}
\]
where $a$ denotes the action (and the same cartesian diagram for the projection $p$). Since $a$ and $p$ are smooth, the lemma follows from smooth base change for equivariant complexes. The general case follows by induction.
\end{proof}

\subsection{Statement of Braden's theorem} 
Let $X/S$ be space, and let $D(X,\La)$ as above the unbounded derived category of $(X_{\text{\'et}},\La)$-modules with $\La=\bbZ/n$ for some $n>1$ invertible on $S$. Let us denote by $D^+(X,\La)$ the full subcategory of $D(X,\La)$ of bounded below complexes. 

\begin{thm}\label{Bradenthm} Let $S$ be a scheme. Let $X/S$ be a space locally of finite presentation with an \'etale locally linearizable $\bbG_{m}$-action. Then, for $\calA\in D^+(X,\La)^{\Gmon}$, the arrow of $D(X^0,\La)$ defined in \eqref{locmorph}
\[
L_{X/S}^-\calA\longto L^+_{X/S}\calA
\] 
is an isomorphism.
\end{thm}

Theorem \ref{Bradenthm} implies Theorem B i) from the introduction. The proof follows Braden's original method: we prove the theorem for affine spaces with a linear $\bbG_m$-action, cf. \S \ref{moncomplexes}, \S \ref{secconlem}, \S \ref{linearactions} below, and reduce to the latter case using an equivariant atlas, cf. \S \ref{compclosedsec}, \S \ref{competalesec} below. Let us warm up with the case of $\bbA^1$.
\subsection{Monodromic complexes on $\bbA^1$}\label{moncomplexes}
Let $q\co \bbA^1_S\to S$ denote the structure morphism. Let $i\co S\to \bbA^1_S$ be any morphism of $S$-schemes such that $q\circ i=\id$. As in \eqref{contrafos} above, there are natural transformations 
\begin{equation}\label{easytrafo}
q_*\longto i^*, \;\;\;\text{(resp. $i^!\longto q_!$)} 
\end{equation}
of functors from $D(\bbA^1_S,\La)$ to $D(S,\La)$.

\begin{lem} \label{seclem}
Let $\calB$ be a sheaf of $\La$-modules on $S$. The transformation \eqref{easytrafo} is an isomorphism for $q^*\calB$ (resp. $q^!\calB$).
\end{lem}
\begin{proof} If $S$ is the spectrum of a separably closed field, the map
\begin{equation}\label{seclem1}
q_*q^*\calB\overset{\simeq}{\longto}i^*q^*\calB=\calB
\end{equation}
is an isomorphism: $\calB$ is of torsion invertible on $S$ which implies that $H^i(\bbA^1_{S},q^*\calB)=0$ for $i>0$. In the general case, we may assume $S$ to be the spectrum of a strictly Henselian local ring, and \eqref{seclem1} follows from local acyclicity of smooth morphisms, cf. [SGA4$1\over 2$, Arcata, \S V, Thm. 1.7]\footnote{The noetherian assumption in [SGA4$1\over 2$] can be removed as follows. If $S=\lim_iS_i$ with $S_i$ noetherian and affine, then on constructible sheaves $\text{Sh}_c(S)=\colim_i\text{Sh}_c(S_i)$ by [SGA 4, IX Cor. 2.7.4]. By compatibility with filtered colimits \eqref{seclem1} holds for constructible sheaves with $S$ arbitrary. As every torsion sheaf is a colimit of constructible sheaves, we obtain \eqref{seclem1} in general.}. In the other case, the map\begin{equation}\label{seclem2}
\calB=q^!i^!\calB\overset{\simeq}{\longto}q_!q^!\calB
\end{equation}
is an isomorphism. Indeed, by \cite[Thm. 0.1.4 (2)]{LZ12}, we have $q^!=q^*\langle1\rangle$ with $\langle1\rangle=[2](1)$, and base change reduces us to the case that $S$ is the spectrum of a separably closed field. In this case, $H^{2-i}_c(\bbA^1_S,q^*\calB(1))$ is zero for $i\not= 0$.
\end{proof}

\begin{cor}\label{seccor1} Let $\calB\in D^+(S,\La)$. Then 
\[
q_*(j_!j^*q^*\calB)=0 \;\;\;\;\text{(resp. $q_!(j_*j^*q^!\calB)=0$),}  
\]
where $j\co \bbA^1_S\bslash i(S)\to \bbA^1_S$ is the inclusion. 
\end{cor}
\begin{proof} Using the sheafy version of the Leray spectral sequence, we may assume that $\calB$ is a sheaf of $\La$-modules. Applying $q_*$ (resp. $q_!$) to the distinguished triangle $j_!j^*(q^*\calB)\to q^*\calB \to i_*i^*(q^*\calB)\to $ (resp. $i_!i^!(q^!\calB)\to q^!\calB \to j_*j^*(q^!\calB)\to $) we see that the corollary follows from Lemma \ref{seclem}.
\end{proof}

Let $i_{0}\co S\to \bbA^1_S$ denote the zero section. 

\begin{cor} \label{seccor2}
Let $\calA\in D^+(\bbA^1_S,\La)^{\Gmon}$ where $\bbA^1_S$ is equipped with the standard action. The transformation $q_*\to i^*_{0}$ (resp. $i^!_{0}\to q_!$) is an isomorphism for $\calA$.
\end{cor}
\begin{proof}
By induction we may assume that $\calA$ is $\bbG_m$-equivariant. Let $j\co \bbG_{m,S}\to \bbA^1_S$ be the inclusion. Then the complex $j^*\calA$ is of the form $j^*q^*\calB$ for a complex $\calB\in D^+(S,\La)$. Hence, Corollary \ref{seccor1} implies the claim.
\end{proof}

Let $i_{1}\co S\to \bbA^1_S$ be the unit section.

\begin{cor}\label{seccor3} Let $\calB\in D^+(S,\La)$. Let $\varphi\co q^*\calB\to q^*\calB$ (resp. $q^!\calB\to q^!\calB$) be a map such that $i^*_{1}\varphi$ (resp. $i^!_{1}\varphi$) is an isomorphism and $i^*_{0}\varphi$ (resp. $i^!_{0}\varphi$) is zero. Then $\calB=0$.
\end{cor}
\begin{proof}
Since $q^*\calB$ is $\bbG_m$-monodromic, there is a commutative (up to natural isomorphism) diagram in $D^+(S,\La)$
\begin{equation}
\begin{tikzpicture}[baseline=(current  bounding  box.center)]
\matrix(a)[matrix of math nodes, 
row sep=1.5em, column sep=2em, 
text height=1.5ex, text depth=0.45ex] 
{i^*_{1}q^*\calB & q_*q^*\calB & i^*_{0}q^*\calB \\ 
i^*_{1}q^*\calB & q_*q^*\calB & i^*_{0}q^*\calB, \\ }; 
\path[->](a-1-1) edge node[left] {$\simeq$} (a-2-1);
\path[->](a-1-2) edge node[above] {$\simeq$} (a-1-1);
\path[->](a-2-2) edge node[above] {$\simeq$} (a-2-1);
\path[->](a-1-2) edge node[above] {$\simeq$} (a-1-3); 
\path[->](a-1-2) edge  (a-2-2); 
\path[->](a-2-2) edge node[above] {$\simeq$} (a-2-3);
\path[->](a-1-3) edge node[right] {$0$} (a-2-3);
\end{tikzpicture}
\end{equation}
where the vertical maps are induced by $\varphi$. Hence, the zero map on $\calB$ is an isomorphism, i.e. $\calB=0$. The $!$-case follows similarly.
\end{proof}

\subsection{Braden's contraction lemma}\label{secconlem} The core of Theorem \ref{Bradenthm} for affine spaces with a linear $\bbG_m$-action is an analogue of Braden's contraction lemma \cite[Lem. 6]{Br03}. We follow his method which is based on arguments of Springer \cite{Sp84}.\smallskip\\
Let $\calE$ be a locally free $\calO_S$-module of finite rank. Let $\bbV(\calE)=\Spec_S(\text{Sym}^\otimes(\calE))$ be the associated vector bundle over $S$. A linear $\bbG_m$-action on $\bbV(\calE)$ is equivalent to a morphism of group schemes $\bbG_{m,S}\to \GL(\calE)$. If $S$ is connected such a morphism corresponds to a $\bbZ$-grading on $\calE$, i.e. as $\calO_S$-modules 
\begin{equation}\label{gradedvec}
\calE=\oplus_{i\in\bbZ}\calE_i,
\end{equation}
where the decomposition is according to the weights of the $\bbG_m$-action. Let us fix a non-trivial decomposition 
\begin{equation}\label{wdecom}
\calE=\calE^+\oplus\calE^-,
\end{equation}
such that $\calE^+=\oplus_{i\geq k}\calE_{i}$ and $\calE^-=\oplus_{i<k}\calE_i$ for some fixed $k\in\bbZ$. Let $\bbP(\calE)=\text{Proj}_S(\text{Sym}^\otimes(\calE))$ be the corresponding projective bundle over $S$. Then $Z=\bbP(\calE^-)$, and $Y=\bbP(\calE)\bslash \bbP(\calE^+)$ are $S$-schemes equipped with $\bbG_m$-actions. The decomposition \eqref{wdecom} gives $\bbG_m$-equivariant maps
\begin{equation}
Z\overset{\iota}{\longto} Y\overset{\pi}{\longto} Z
\end{equation}
with $\pi\circ \iota=\id$ and $Z^0=Y^0$. Note that the last equality already follows from all weights of $\calE^+$ being different from all weights of $\calE^-$.


\begin{rmk} Recall that $\bbP(\calE)$ represents the functor that to an $S$-scheme $T$ associates the set of locally direct summand quasi-coherent $\calO_T$-submodules $\calF\subset \calE_T$  such that $\calE_T/\calF$ is locally free of rank $1$. In this description the map $\iota$ is given by $\calF\mapsto \calF\oplus\calE^+_T$ and $\pi$ is given by $\calF\mapsto \calF\cap\calE_T^-$. To check that $\pi$ is well-defined note that for any locally direct summand $\calF\subset \calE_T$ as above not containing $\calE^-_T$ the natural map $\calE_T^-\to \calE_T/\calF$ is surjective (use Nakayama) and gives an isomorphism $\calE_T^-/(\calE^-_T\cap\calF)\simeq \calE_T/\calF$.
\end{rmk}

Let us denote by $\tau\co Z\to S$ the structure morphism. Applying $\tau_*$ (resp. $\tau_!$) to the natural transformation $\pi_*\to \iota^*$ (resp. $\iota^!\to \pi_!$) gives
\begin{equation}\label{bradentrafos}
\tau_*\pi_*\longto \tau_*\iota^* \;\;\;\;\;\text{(resp. $\tau_!\iota^!\to \tau_!\pi_!$)}
\end{equation}
as natural transformations from $D(Y,\La)$ to $D(S,\La)$. Note that $\tau$ is proper and thus $\tau_*=\tau_!$.  

\begin{prop}[Braden's contraction lemma] \label{contractionlem} 
The transformations \eqref{bradentrafos} restricted to the category $D^+(Y,\La)^{\Gmon}$ are isomorphisms.
\end{prop}

\begin{rmk}\label{inverse} Let us point out that the proposition also holds true with the roles of $\calE^+$ and $\calE^-$ interchanged. Indeed, any complex is $\bbG_m$-monodromic if and only if it is $\bbG_m$-monodromic for the inverse $\bbG_m$-action. 
\end{rmk}

We first prove the proposition for $\tau_*\pi_*\longto \tau_*\iota^*$, and then explain the adjustments for the $!$-case in Remark \ref{shriekadjustments} below. Of course, in the presence of duality both transformations are dual to each other. Let us start.

 Let $\bbG_{m,S}\times Y\to \bbG_{m,S}\times Y\times Y$,  $(\la,y)\mapsto (\la,y,\la\cdot y)$ be the graph of the action map (closed because $Y/S$ is separated), and denote by $\Ga$ its scheme theoretic closure in $\bbA^1_S\times Y\times Y$. Let $p_1\co \Ga\to \bbA^1_S\times Y$, $(y_1,y_2,y_3)\mapsto (y_1,y_2)$ and let $p_2\co \Ga\to \bbA^1_S\times Y$, $(y_1,y_2,y_3)\mapsto (y_1,y_3)$.  

\begin{lem} \label{ebasiclem}
i) The morphisms $p_1, p_2\co \Ga\to \bbA^1_S\times Y$ are isomorphisms over $\bbG_{m,S}\times Y$.\smallskip\\
ii) The morphism $p_1\co \Ga\to \bbA^1_S\times Y$ is proper.
\end{lem}

\begin{rmk} For part ii), the hypothesis on the weights in \eqref{wdecom} is crucial.
\end{rmk}

\begin{proof}[Proof of Lemma \ref{ebasiclem}] 
Part i) follows from the following elementary fact. If $I\subset R[t,t^{-1}]$ is an ideal for some ring $R$, then $I=(I\cap R[t])\cdot R[t,t^{-1}]$ by flatness of localizations. We claim that $\Ga$ is a closed subscheme of $\bbA^1_S\times Y\times \bbP(\calE)$. Consider the cartesian diagram of $S$-schemes
\[
\begin{tikzpicture}[baseline=(current  bounding  box.center)]
\matrix(a)[matrix of math nodes, 
row sep=1em, column sep=2em, 
text height=1.5ex, text depth=0.45ex] 
{\bbG_{m,S}\times Y\times Y & \bbG_{m,S}\times Y\times \bbP(\calE) \\ 
\bbA^1_S\times Y\times Y & \bbA^1_S\times Y\times \bbP(\calE), \\}; 
\path[->](a-1-1) edge node[above] {} (a-1-2);
\path[->](a-1-1) edge node[left] {} (a-2-1);
\path[->](a-1-2) edge node[right] {} (a-2-2);
\path[->](a-2-1) edge node[above] {} (a-2-2);
\end{tikzpicture}
\]
where all maps are open immersions. Let $\Ga_a$ denote the graph of the action map $a\co \bbG_{m,S}\times Y\to Y$. Then $\Ga_a$ is a closed subscheme of both $\bbG_{m,S}\times Y\times Y$ and $ \bbG_{m,S}\times Y\times \bbP(\calE)$ (because $\Ga_a$ is the restriction of the graph of the full action $\bbG_{m,S}\times\bbP(\calE)\to \bbP(\calE)$ to $\bbG_{m,S}\times Y\times \bbP(\calE)$). Let $\tilde{\Ga}$ be the closure of $\Ga_a$ in $\bbA^1_S\times Y\times \bbP(\calE)$. Since $\bbG_{m,S}\subset \bbA^1_S$ is quasi-compact, we have by \cite[Tag 01R5, Lem. 28.6.3]{StaPro} that $\Ga=\tilde{\Ga}\cap(\bbA^1_S\times Y\times Y)$. Hence, $\Ga\to\tilde{\Ga}$ is an open immersion, and it is enough to show that $|\Ga|\to |\tilde{\Ga}|$ is surjective on topological spaces. For this we may assume $S$ to be the spectrum of an algebraically closed field. It is enough to show that $|\Ga|\to |\tilde{\Ga}|$ is surjective over points lying in the zero section of $\bbA^1_S$. An argument on coordinates implies that whenever $[y^+:y^-]\in Y(S)$ (i.e. $y^-\not = 0$) is a point, then its limit $\lim_{\la\to 0}\la\cdot [y^+:y^-]$ in $\bbP(\calE)(S)$ is of the form $[0:{'y}^-]\in Y(S)$ (because of the weight hypothesis in \eqref{wdecom}). This shows surjectivity of $|\Ga|\to |\tilde{\Ga}|$ and proves the lemma.
\end{proof}

Consider the following diagram of $S$-schemes with the square being cartesian
\begin{equation}\label{notationdiag}
\begin{tikzpicture}[baseline=(current  bounding  box.center)]
\matrix(a)[matrix of math nodes, 
row sep=1.5em, column sep=2em, 
text height=1.5ex, text depth=0.45ex] 
{\Ga& \bbA^1_S\times Y & \\ 
\bbA^1_S&  & Y\\
& S& \\}; 
\path[->](a-1-1) edge node[below] {$p_2$\;\;} node[above] {$p_1$\;\;} (a-1-2);
\path[->](a-1-2) edge node[below] {\;$q_1$} (a-2-1);
\path[->](a-1-2) edge node[below] {$q_2$\;} (a-2-3);
\path[->](a-2-1) edge node[below] {$q$\;\;} (a-3-2);
\path[->](a-2-3) edge node[below] {\;\;\;\;\;\;\;$\tau\circ \pi$} (a-3-2);
\end{tikzpicture}
\end{equation}
where $q_1$, $q_2$ denote the projections. Let $Z\overset{\iota}{\to} Y \overset{\sig}{\leftarrow} U$ where $U=Y/Z$ is the open complement.

\begin{lem} \label{Bradeniso} Let $\calA\in D^+(Y,\La)$ be $\bbG_m$-equivariant. Then there is an isomorphism in $D^+(\bbA^1_S,\La)$
\[
q_{1,*}p_{2,*}p_2^*q_2^* (\sig_!\sig^*\calA)\;\simeq\;  j_!j^*q^*\tau_*\pi_*(\sig_!\sig^*\calA)
\]
 where $j\co \bbG_{m,S}\to \bbA^1_S$ denotes the inclusion.
\end{lem}
\begin{proof} Let $\tilde{\calA}=q_{1,*}p_{2,*}p_2^*q_2^* (\sig_!\sig^*\calA)$. By smooth base change $j^*\tilde{\calA}\simeq j^*q^*\tau_*\pi_*(\sig_!\sig^*\calA)$ because $p_2$ is an isomorphism over $\bbG_{m,S}\times Y$ by Lemma \ref{ebasiclem} i). Let $i\co S\to \bbA^1_S$ be the zero section. By considering the distinguished triangle $j_!j^*\tilde{\calA}\to \tilde{\calA}\to i_*i^*\tilde{\calA}\to $, we are reduced to show $i^*\tilde{\calA}=0$. Since $q_1\circ p_1=q_1\circ p_2$, we get
\begin{equation}\label{longiso1}
\tilde{\calA}\;\simeq\; q_{1,*}p_{1,!}p_2^*q_2^* (\sig_!\sig^*\calA)
\end{equation}
because $p_1$ is proper, cf. Lemma \ref{ebasiclem} ii). There is a cartesian diagram of $S$-schemes
\begin{equation}\label{squareforproof}
\begin{tikzpicture}[baseline=(current  bounding  box.center)]
\matrix(a)[matrix of math nodes, 
row sep=1.5em, column sep=2em, 
text height=1.5ex, text depth=0.45ex] 
{\bbG_{m,S}\times U & \Ga \\ 
U & Y, \\}; 
\path[->](a-1-1) edge node[above] {$\sig'$} (a-1-2);
\path[->](a-1-1) edge node[left] {$a$} (a-2-1);
\path[->](a-1-2) edge node[right] {$q_2\circ p_2$} (a-2-2);
\path[->](a-2-1) edge node[above] {$\sig$} (a-2-2);
\end{tikzpicture}
\end{equation}
where $\sig'\co \bbG_{m,S}\times U\to \Ga$, $(\la,u)\mapsto (\la, u, \la\cdot u)$ and $a\co \bbG_{m,S}\times U\to U$, $(\la,u)\mapsto \la\cdot u$ is the action (to prove cartesian recall $Y^0=Z^0$).  
By \eqref{longiso1}, the complex $\tilde{\calA}$ becomes
\begin{equation}\label{longiso2}
 q_{1,*}p_{1,!}(p_2^*q_2^* \sig_!)\sig^*\calA \;\simeq\; q_{1,*}p_{1,!}(\sig'_!a^*)\sig^*\calA\;\simeq\; q_{1,*}p_{1,!}\sig'_!(a^*\sig^*\calA).
\end{equation}
Since $\calA$ and hence $\sig^*\calA$ are $\bbG_m$-equivariant, we have $a^*(\sig^*\calA)\simeq (q'_2)^*(\sig^*\calA)$ where $q_2'=q_2|_{\bbG_{m,S}\times U}\co \bbG_{m,S}\times U\to U$ is the projection. Moreover, $p_1\circ \sig'=j\times \sig$, and \eqref{longiso2} becomes
\begin{equation}\label{longiso3}
q_{1,*}p_{1,!}\sig'_!((q'_2)^*\sig^*\calA)\;\simeq\; q_{1,*}(j\times\sig)_!(q'_2)^*\sig^*\calA\;\simeq\; q_{1,*}j_!j^*(q_2^*\sig_!\sig^*\calA)
\end{equation}
where $j\co \bbG_{m,S}\times Y\to \bbA^1_S\times Y$ denotes the inclusion (by abuse of notation). By Corollary \ref{seccor2}\footnote{Apply the corollary to the $\bbG_m$-structure with respect to $\bbA^1_S$ ignoring the action on $Y$.} and \eqref{longiso3} we get 
\begin{equation}\label{longiso4}
i^*\tilde{\calA}\;\simeq\; i^*q_{1,*}j_!j^*(q_2^*\sig_!\sig^*\calA)\;\simeq\; q_*q_{1,*}j_!j^*(q_2^*\sig_!\sig^*\calA)
\end{equation}
because $q_{1,*}j_!j^*(q_2^*\sig_!\sig^*\calA)$ is $\bbG_m$-monodromic for the standard action on $\bbA^1_S$. Since $q\circ q_1=(\tau\circ \pi)\circ q_2$, we get for \eqref{longiso4} that
\[
i^*\tilde{\calA}\;\simeq\; (\tau\circ \pi)_* q_{2,*}j_!j^*(q_2^*\sig_!\sig^*\calA)\;=\;0.
\]
because $q_{2,*}j_!j^*(q_2^*\sig_!\sig^*\calA)=0$ by Corollary \ref{seccor1}.


\end{proof}

\begin{proof}[Proof of Braden's contraction lemma.] Let $\calA\in D^+(Y,\La)^{\Gmon}$. By induction we may assume that $\calA$ is $\bbG_m$-equivariant. Let $Z\overset{\iota}{\to}Y\overset{\sig}{\leftarrow}U$ as above, and apply $\tau_*\pi_*$ to the associated distinguished triangle
\[
\sig_!\sig^*\calA\to \calA\to \iota_*\iota^*\calA\to.
\] 
We have to show that $\calB=\tau_*\pi_*(\sig_!\sig^*\calA)$ vanishes. Let $q\co\bbA^1_S\to S$ be the structure morphism. Our aim is to construct a morphism
\[
\varphi\co q^*\calB\longto q^* \calB
\] 
in $D^+(\bbA^1_S,\La)$ which is an isomorphism when restricted to the unit section and zero when restricted to the zero section. Then Corollary \ref{seccor3} implies that $\calB=0$. Let us construct $\varphi$. Note that $q^* \calB\;\simeq\; q_{1,*}q_2^*(\sig_!\sig^*\calA)$ by smooth base change applied to \eqref{notationdiag}. The unit $\id\to p_{2,*}p_2^*$ gives a transformation
\begin{equation}\label{endocons}
q_{1,*}q_2^*(\sig_!\sig^*\calA) \longto\; q_{1,*}p_{2,*}p_2^*q_2^*(\sig_!\sig^*\calA).
\end{equation}  
Since $p_2$ is an isomorphism over $\bbG_{m,S}\times Y$ (cf. Lemma \ref{ebasiclem}), it follows that \eqref{endocons} is an isomorphism restricted to $\bbG_{m,S}$. By Lemma \ref{Bradeniso}, we have
\[
q_{1,*}p_{2,*}p_2^*q_2^*(\sig_!\sig^*\calA)\;\simeq\; j_!j^*q^*\tau_*\pi_*(\sig_!\sig^*\calA)\;=\;j_!j^*q^*\calB,
\]
where $j\co \bbG_{m,S}\to \bbA^1_S$ is the inclusion. Composing \eqref{endocons} with the adjunction $j_!j^*q^*\calB\to q^*\calB$ constructs the desired morphism $\varphi$. This proves the proposition.
\end{proof}

\begin{rmk}\label{shriekadjustments} The adjustments for the $!$-case are as follows. By considering the triangle $\iota_!\iota^!\calA\to\calA\to \sig_*\sig^*\calA\to$ it is enough to show that $\calB=\tau_!\pi_!(\sig_*\sig^*\calA)$ vanishes. The aim is to construct a map $\varphi\co q^!\calB\to q^!\calB$ such that $j^*\varphi$ is an isomorphism and $\varphi$ vanishes when $!$-restricted to the zero section. Then Corollary \ref{seccor3} implies that $\calB=0$. As above $q^!\calB\simeq q_{1,!}q_2^!(\sig_*\sig^*\calA)$ by smooth base change. Now the counit $p_{2,!}p_2^!\to \id$ gives a map
\begin{equation}\label{shriekcase}
q_{1,!}p_{2,!}p_2^!q_2^!(\sig_*\sig^*\calA)\to q_{1,!}q_2^!(\sig_*\sig^*\calA),
\end{equation}
and as in Lemma \ref{Bradeniso}, one shows that $q_{1,!}p_{2,!}p_2^!q_2^!(\sig_*\sig^*\calA)\simeq j_*j^*q^!\calB$. Precomposing \eqref{shriekcase} with $q^!\calB\to j_*j^*q^!\calB$ constructs the desired map $\varphi$.
\end{rmk}

\subsection{Linear actions}\label{linearactions} Our argument follows Braden's argument in \cite{Br03}. Let us explain how the contraction lemma (Proposition \ref{contractionlem} above) implies Theorem \ref{Bradenthm} for affine spaces. \smallskip\\
Let $S$ be connected, and let $\calE$ be a locally free  $\calO_S$-module of finite rank with $\bbG_m$-action. Consider the weight decomposition
\begin{equation}\label{lindecom}
\calE=\calE^+\oplus\calE^-\oplus\calE^0,
\end{equation}
where $\calE^0=\calE_0$ is the zero component in the weight decomposition \eqref{gradedvec}, and $\calE^+=\oplus_{i> 0}\calE_i$ and $\calE^-=\oplus_{i< 0}\calE_i$. By the explicit description in \S \ref{theaffinecase}, the hyperbolic localization diagram $\HypLoc{\bbV(\calE)}$ becomes\footnote{Note that $j\co \bbV(\calE^0)\to \bbV(\calE^+\oplus \calE^0)\times_{\bbV(\calE)}\bbV(\calE^-\oplus\calE^0)$ is an isomorphism in this case.}
\begin{equation}\label{affinehyperloc}
\begin{tikzpicture}[baseline=(current  bounding  box.center)]
\matrix(a)[matrix of math nodes, 
row sep=1.5em, column sep=2em, 
text height=1.5ex, text depth=0.45ex] 
{\bbV(\calE^0)&\bbV(\calE^-\oplus\calE^0) & \bbV(\calE^0)\\ 
\bbV(\calE^+\oplus\calE^0)& \bbV(\calE) & \\
\bbV(\calE^0),& &\\}; 
\path[->](a-1-1) edge node[above] {$i^-$} (a-1-2);
\path[->](a-2-1) edge node[below] {$p^+$} (a-2-2);
\path[->](a-1-1) edge node[left] {$i^+$} (a-2-1);
\path[->](a-1-2) edge node[right] {$p^-$} (a-2-2);
\path[->](a-1-2) edge node[above] {$q^-$} (a-1-3);
\path[->](a-2-1) edge node[left] {$q^+$} (a-3-1);
\end{tikzpicture}
\end{equation}
where all maps are induced by the decomposition \eqref{lindecom}. 

\begin{prop} \label{linactprop} 
Let $S$ be a connected scheme. Then Theorem \ref{Bradenthm} holds for $X=\bbV(\calE)$ with a linear $\bbG_m$-action.
\end{prop}

As a benefit of working over a general base $S$, we may and do assume that $\bbV(\calE^0)=S$, i.e. $\calE^0=0$. 

\begin{lem}\label{contlem} Let $\calA$ be a $\bbG_m$-monodromic bounded below complex of $\La$-modules on $\bbV(\calE^-)$ (resp. $\bbV(\calE^+)$). Then the transformation
\[
(q^-)_*\calA\overset{\simeq}{\longto} (i^-)^*\calA\;\;\;\text{(resp. $(i^+)^!\calA\overset{\simeq}{\longto} (q^+)_!\calA$)}
\]
is an isomorphism.
\end{lem}
\begin{proof} Apply Proposition \ref{contractionlem} to $\calO_S\oplus\calE^-$ (resp. $\calE^+\oplus \calO_S$) where $\calO_S$ is of weight $0$. Then $Z=S$ and $Y=\bbV(\calE^-)$ using the inverse $\bbG_m$-action (resp. $Y=\bbV(\calE^+)$).
\end{proof}

\begin{rmk}
i) Similar lemmas are well known in different contexts, cf. \cite[Prop. 5.3.2]{DG11} and the references cited there. \smallskip\\
ii) Let us sketch a direct proof of Lemma \ref{contlem} which is independent of Proposition \ref{contractionlem}. Blowing up the zero section in $\bbV(\calE^\pm)$ one reduces to the case of ${^\chi\bbA}^1_S$ where $\bbG_m$-acts through a character $\chi\co \bbG_m\to \bbG_m$, $\la\mapsto \la^a$ with $a\not =0$. Inverting the $\bbG_m$-action if necessary we may assume that $a>0$. Consider the $\bbG_m$-equivariant finite flat map
\[
\pi\co\bbA^1_S\longto {^\chi\bbA}^1_S,\;\; x\longmapsto x^a,
\]
where $\bbA^1_S$ is equipped with the standard action. If $a$ in invertible on $S$, then $\pi|_{\bbG_{m,S}}$ is \'etale. If $S$ is an $\bbF_p$-scheme and $a=p$, then $\pi$ is the relative Frobenius. A case analysis shows that $q_*\calA\simeq q_*\pi_*\pi^*\calA$ and $q_!\calA\simeq q_!\pi_!\pi^!\calA$ where $q$ denotes the structure morphism. Hence, we may reduce to $\bbA^1_S$ with the standard action. Then Corollary \ref{seccor2} implies the lemma. 
\end{rmk}

\begin{proof}[Proof of Proposition \ref{linactprop}.] Note $\calE=\calE^+\oplus \calE^-$ (under the assumption $\calE^0=0$). If either $\calE^+=0$ or $\calE^-=0$, then Proposition \ref{linactprop} reduces to Lemma \ref{contlem} above. Hence, we assume the decomposition to be non-trivial. Let $\calA\in D^+(\bbV(\calE),\La)^{\Gmon}$. Let $\bbV(\calE^+)\overset{p^+}{\to}\bbV(\calE)\overset{j}{\leftarrow}\bbV(\calE)\bslash\bbV(\calE^+)$, and consider the associated distinguished triangle
\begin{equation}\label{linactprop1}
j_!j^*\calA\longto \calA\longto (p^+)_*(p^{+})^*\calA\longto.
\end{equation}
Applying $(i^-)^*(p^-)^!$ to \eqref{linactprop1} the right arrow becomes 
\begin{equation}\label{linactprop2}
(i^-)^*(p^-)^!\calA\;\longto\;  (i^-)^*(p^-)^! (p^+)_*(p^{+})^*\calA \;\simeq\;  (i^+)^!(p^{+})^*\calA,
\end{equation}
because the square in \eqref{affinehyperloc} is cartesian.
By Lemma \ref{contlem} it is enough to show that 
\[
(i^-)^*(p^-)^!(j_!j^*\calA)\;=\;0. 
\]
Consider the direct sum $(\calE^+\oplus\calO_S)\oplus\calE^-$ where $\calO_S$ is of weight $0$. Let $Z=\bbP(\calE^+\oplus\calO_S)$, and denote $Y=\bbP(\calE\oplus \calO_S)\bslash \bbP(\calE^-)$. There is a $\bbG_m$-equivariant diagram of $S$-schemes
\begin{equation}\label{linactprop3}
\begin{tikzpicture}[baseline=(current  bounding  box.center)]
\matrix(a)[matrix of math nodes, 
row sep=1.5em, column sep=2em, 
text height=1.5ex, text depth=0.45ex] 
{\bbV(\calE^+)& \bbV(\calE) &  \bbV(\calE^-)& \\
Z & Y & Z & S,\\}; 
\path[->](a-1-1) edge node[above] {\;$p^+$} (a-1-2);
\path[->](a-1-3) edge node[above] {\;$p^-$} (a-1-2);
\path[->](a-1-3) edge node[above] {\;\;\;\;$q^-$} (a-2-4);
\path[->](a-1-1) edge node[right] {} (a-2-1);
\path[->](a-2-1) edge node[above] {$\iota$} (a-2-2);
\path[->](a-1-2) edge node[right] {$\rho$} (a-2-2);
\path[->](a-2-2) edge node[above] {$\pi$} (a-2-3);
\path[->](a-2-3) edge node[above] {$\tau$} (a-2-4);
\end{tikzpicture}
\end{equation}
where $\rho$ is an open immersion, and $i'=\rho\circ p^-$ is a closed immersion. Let $j'\co Y\bslash\bbV(\calE^-)\to Y$ be the open complement. Let $\calB=\rho_!(j_!j^*\calA)$, and consider the distinguished triangle 
\begin{equation}\label{linactprop4}
i'_*(i')^!\calB\longto\calB \longto j'_*(j')^*\calB\longto.
\end{equation}
Applying $(\tau\circ \pi)_*$ the first term in \eqref{linactprop4} becomes
\[
(\tau\circ \pi)_*i'_*(i')^!\calB\;\simeq\; (\underbrace{\tau\circ \pi\circ i'}_{=q^-})_*(p^-)^!\underbrace{\rho^!\rho_!}_{\simeq \id}(j_!j^*\calA)  \;\simeq\; (q^-)_*(p^-)^!(j_!j^*\calA)
\]
which is $(i^-)^*(p^-)^!(j_!j^*\calA)$ by Lemma \ref{contlem}. Let us show that $(\tau\circ \pi)_*$ of the second and third term in \eqref{linactprop4} vanishes.\smallskip\\
{\it (1) $(\tau\circ \pi)_*\calB=0$.} There is a $\bbG_m$-equivariant commutative diagram of $S$-schemes
\[
\begin{tikzpicture}[baseline=(current  bounding  box.center)]
\matrix(a)[matrix of math nodes, 
row sep=1.5em, column sep=2em, 
text height=1.5ex, text depth=0.45ex] 
{\bbV(\calE^+)& \bbV(\calE) &  \bbV(\calE)\bslash\bbV(\calE^+) \\
Z & Y & Y\bslash Z. \\}; 
\path[->](a-1-1) edge node[above] {\;$p^+$} (a-1-2);
\path[->](a-1-3) edge node[above] {\;$j$} (a-1-2);
\path[->](a-1-1) edge node[right] {} (a-2-1);
\path[->](a-2-1) edge node[above] {$\iota$} (a-2-2);
\path[->](a-1-2) edge node[right] {$\rho$} (a-2-2);
\path[->](a-2-3) edge node[above] {$\sig$} (a-2-2);
\path[->](a-1-3) edge node[right] {$\rho$} (a-2-3);
\end{tikzpicture}
\]
Hence, $\calB=\rho_!j_!j^*\calA\simeq \sig_!(\rho_!j^*\calA)$. This gives 
\[
\tau_*\pi_*\calB\;\simeq\; \tau_*\pi_*\sig_!(\rho_!j^*\calA)\;\simeq\;\tau_*\iota^*\sig_!(\rho_!j^*\calA)\;=\;0,
\]  
where we apply Proposition \ref{contractionlem} to the $\bbG_m$-monodromic complex $\sig_!(\rho_!j^*\calA)$ with respect to the inverse $\bbG_m$-action, cf. Remark \ref{inverse}. This shows (1).\smallskip\\
{\it (2) $(\tau\circ \pi)_*(j'_*(j')^*\calB)=0$.} There is a $\bbG_m$-equivariant commutative diagram of $S$-schemes
\[
\begin{tikzpicture}[baseline=(current  bounding  box.center)]
\matrix(a)[matrix of math nodes, 
row sep=1.5em, column sep=2em, 
text height=1.5ex, text depth=0.45ex] 
{\emptyset & \bbV(\calE)\bslash \bbV(\calE^+) &  \bbV(\calE)\bslash (\bbV(\calE^+)\cup \bbV(\calE^-)) \\
Z & Y & Y\bslash \bbV(\calE^-). \\}; 
\path[->](a-1-1) edge node[above] {} (a-1-2);
\path[->](a-1-3) edge node[above] {$j'$} (a-1-2);
\path[->](a-1-1) edge node[right] {} (a-2-1);
\path[->](a-2-1) edge node[above] {$\iota$} (a-2-2);
\path[->](a-1-2) edge node[right] {$\rho\circ j$} (a-2-2);
\path[->](a-2-3) edge node[above] {$j'$} (a-2-2);
\path[->](a-1-3) edge node[right] {$\rho\circ j$} (a-2-3);
\end{tikzpicture}
\] 
where the squares are cartesian. Hence, $(j')^*\calB=(j')^*(\rho\circ j)_!j^*\calA\simeq (\rho\circ j)_!((j')^*j^*\calA)$ by base change. Now $Y\bslash \bbV(\calE^-)=\bbP(\calE\oplus\calO_S)\bslash \bbP(\calE^-\oplus \calO_S)$, and we may consider
\[
\bbP(\calE^+)\overset{\iota'}{\longto}Y\bslash\bbV(\calE^-)\overset{\pi'}{\longto}\bbP(\calE^+)\overset{\tau'}{\longto} S.
\]
Now $\tau\circ \pi\circ j'=\tau'\circ\pi'$ which gives 
\begin{equation}\label{thelastterm}
\tau_* \pi_*j'_*(j')^*\calB\;\simeq\; \tau'_*\pi'_*(\rho\circ j)_!((j')^*j^*\calA)\;\simeq\;\tau'_*(\iota')^*(\rho\circ j)_!((j')^*j^*\calA),
\end{equation}
where we apply Proposition \ref{contractionlem} to the decomposition $\calE^+\oplus(\calO_S\oplus\calE^-)$ and the $\bbG_m$-monodromic complex $(\rho\circ j)_!(j')^*j^*\calA$ (again by considering the inverse of the $\bbG_m$-action). But the last term in \eqref{thelastterm} vanishes because $(\rho\circ j)_!((j')^*j^*\calA)$ lives on $\bbV(\calE)=\bbP(\calE\oplus\calO_S)\bslash \bbP(\calE)$. This shows (2) and proves the proposition.
\end{proof}

\subsection{The affine case}\label{compclosedsec}
Let us explain how Proposition \ref{linactprop} implies Theorem \ref{Bradenthm} (Zariski locally on $S$) for $S$-affine schemes $X$ of finite presentation with $\bbG_m$-action.
 
\begin{lem}\label{embedlem}
Let $S$ be a scheme, and let $X/S$ be a $S$-affine scheme of finite presentation with $\bbG_m$-action. Then, for some $n\geq 0$, there exists Zariski locally on $S$ a $\bbG_m$-equivariant closed immersion $X\to \bbA^n_S$, where $\bbG_m$-acts linear on $\bbA^n_S$.
\end{lem}
\begin{proof} Let $S=\Spec(R)$ and $X=\Spec(B)$ be affine. If $S$ is connected, then the assertion of the lemma is equivalent to the existence of a $\bbZ$-graded free $R$-module $E$ of some finite rank $n$ together with a morphism of $\bbZ$-graded $R$-algebras
\[
\text{Sym}^{\otimes}(E)\longto B,
\] 
which is surjective. Let $\{b_i\}_{i\in I}$ be a family of homogenous generators of the $R$-algebra $B$. Let $E=\oplus_{i\in I}R$, where the $i$-th component is given the degree $\deg(b_i)$. Then the morphism $\text{Sym}^{\otimes}(E)\to B$ given by $(r_i)_{i\in I}\mapsto \sum_{i\in I}r_i\cdot b_i$ is surjective and $\bbZ$-graded. Since $B$ is of finite type, the set of homogenous generators can be chosen to be finite. This proves the lemma for $S$ locally connected. In general, write $S=\lim_iS_i$, where $S_i$ is the spectrum of a finitely generated $\bbZ$-algebra (in particular locally connected). Since $X/S$ is of finite presentation, it is defined over some $S_i$, and the lemma follows in general.
\end{proof}

Since push forward under closed immersions is conservative, we are reduced to:

\begin{lem}\label{closedlem}
Let $f\co X\to Z$ be a $\bbG_m$-equivariant closed immersion of $S$-affine schemes of finite presentation. Then there is a commutative (up to natural isomorphism) diagram of transformation of functors from $D(X,\La)$ to $D(Z^0,\La)$ as follows
\[
\begin{tikzpicture}[baseline=(current  bounding  box.center)]
\matrix(a)[matrix of math nodes, 
row sep=1.5em, column sep=2em, 
text height=1.5ex, text depth=0.45ex] 
{(f^0)_*\circ L^-_{X/S}&(f^0)_*\circ L^+_{X/S}\\ 
L^-_{Z/S}\circ f_*& L^+_{Z/S}\circ f_*,\\}; 
\path[->](a-1-1) edge node[above] {} (a-1-2);
\path[->](a-2-1) edge node[below] {} (a-2-2);
\path[->](a-1-1) edge node[left] {$\simeq$} (a-2-1);
\path[->](a-1-2) edge node[right] {$\simeq$}(a-2-2);
\end{tikzpicture}
\]
where the horizontal maps are constructed in \eqref{locmorph}.
\end{lem}

The maps $X^0\overset{i^\pm}{\to}X^\pm\overset{q^\pm}{\to}X^0$ induce by \eqref{contrafos} natural transformations of functors from $D(X^\pm,\La)$ to $D(X^0,\La)$ as follows
\begin{equation}\label{closedbc}
(q^-)_*\longto (i^-)^* \;\;\;\;\text{and}\;\;\;\; (i^+)^!\longto (q^+)_!.
\end{equation}

\begin{lem} \label{complemtwo}
Let $X$ and $Z$ be $S$-affine schemes of finite presentation, and let $f\co X\to Z$ be a $\bbG_m$-equivariant closed immersion. There are commutative (up to natural isomorphism) diagrams of transformations of functors from $D(X^\pm,\La)$ to $D(Z^0,\La)$ 
\[
\begin{tikzpicture}[baseline=(current  bounding  box.center)]
\matrix(a)[matrix of math nodes, 
row sep=1.5em, column sep=2em, 
text height=1.5ex, text depth=0.45ex] 
{(q^-)_*(f^-)_*&(i^-)^*(f^-)_* && (i^+)^!(f^+)_*&(q^+)_!(f^+)_*\\ 
(f^0)_*(q^-)_*&(f^0)_*(i^-)^*&& (f^0)_*(i^+)^! &(f^0)_*(q^+)_!,\\}; 
\path[->](a-1-1) edge node[above] {} (a-1-2);
\path[->](a-2-1) edge node[right] {} (a-2-2);
\path[->](a-1-1) edge node[left] {$\simeq$} (a-2-1);
\path[->](a-1-2) edge node[right] {$\simeq$}(a-2-2);
\path[->](a-1-4) edge node[above] {} (a-1-5);
\path[->](a-2-4) edge node[below] {} (a-2-5);
\path[->](a-1-4) edge node[left] {$\simeq$} (a-2-4);
\path[->](a-1-5) edge node[right] {$\simeq$}(a-2-5);
\end{tikzpicture}
\]
where the horizontal arrows are constructed from \eqref{closedbc}.
\end{lem}
\begin{proof} By the explicit description in Lemma \ref{affinecase} the maps $f^0$ and $f^\pm$ are closed immersions, and we have $Z^0=Z^\pm\times_{X^\pm}X^0$. The vertical maps are constructed from proper base change using that $f^\pm_*=f^\pm_!$.
The commutativity of the functor diagrams is straightforward and left to the reader.
\end{proof}

\begin{proof}[Proof of Lemma \ref{closedlem}.] In view of Lemma \ref{complemtwo}, it is enough to show that $f_*$ commutes with the map $(i^-)^{*} (p^-)^!\to (i^+)^! (p^+)^*$ constructed in \eqref{compD1}. There is a commutative diagram of $S$-schemes
\begin{equation}\label{closedlem1}
\begin{tikzpicture}[baseline=(current  bounding  box.center)]
\matrix(a)[matrix of math nodes, 
row sep=1em, column sep=1em, 
text height=1.5ex, text depth=0.45ex] 
{X^0&& X^- &\\
& Z^0 && Z^- \\
X^+&& X &\\ 
&Z^+&& Z. \\}; 
\path[->](a-1-1) edge node[above] {} (a-1-3) edge node[below] {} (a-2-2) edge (a-3-1);
\path[-stealth](a-1-3) edge (a-2-4) edge [densely dotted] (a-3-3);
\path[-stealth](a-3-1) edge node[left] {} (a-4-2) edge [densely dotted] (a-3-3);	
\path[->](a-2-2) edge (a-4-2) edge (a-2-4);						
\path[->](a-2-4) edge (a-4-4);
\path[->](a-4-2) edge (a-4-4);
\path[->](a-3-3) edge [densely dotted] node[left] {}(a-4-4);
\end{tikzpicture}
\end{equation}
Since both $X$ and $Z$ are affine, the explicit description in Lemma \ref{affinecase} shows that all arrows in \eqref{closedlem1} are closed immersions. In particular they are monomorphisms which implies that all squares in \eqref{closedlem1} are cartesian. Again it is straightforward that the following diagram of transformations is commutative (up to natural isomorphism)
\[
\begin{tikzpicture}[baseline=(current  bounding  box.center)]
\matrix(a)[matrix of math nodes, 
row sep=1.5em, column sep=2em, 
text height=1.5ex, text depth=0.45ex] 
{(f^0)_*(i^-)^{*} (p^-)^!&(f^0)_*(i^+)^! (p^+)^*\\ 
(i^-)^{*} (p^-)^! f_*& (i^+)^! (p^+)^* f_*,\\}; 
\path[->](a-1-1) edge node[above] {} (a-1-2);
\path[->](a-2-1) edge node[below] {} (a-2-2);
\path[->](a-1-1) edge node[left] {$\simeq$} (a-2-1);
\path[->](a-1-2) edge node[right] {$\simeq$}(a-2-2);
\end{tikzpicture}
\]
where the vertical maps are constructed from proper base change, cf. the Proof of Proposition \ref{funcprop} below for more details. This proves the proposition.
\end{proof}

\subsection{End of the proof of Theorem \ref{Bradenthm}}\label{competalesec} Let $S$ be a scheme, and let $X/S$ be a space locally of finite presentations with an \'etale locally linearizable $\bbG_m$-action. Let $\{U_i\to X\}_i$ be a $\bbG_m$-equivariant $S$-affine \'etale covering family. Then $\{U_i^0\to X^0\}_i$ is covering by Theorem \ref{repthm} i). By Lemma \ref{etalelem} below, we reduce to the case that $X$ is $S$-affine. Covering $S$ with affine schemes so that the assertion of Lemma \ref{embedlem} holds, and using Lemma \ref{etalelem} again (for open immesrions), Theorem \ref{Bradenthm} follows from the previous section. It remains to show:

\begin{lem}\label{etalelem}
Let $S$ be a scheme, and let $X/S$ be a space locally of finite presentation with an \'etale locally linearizable $\bbG_m$-action. Let $f\co U\to X$ be a $\bbG_m$-equivariant \'etale morphism with $U$ being $S$-affine. Then there is a commutative (up to natural isomorphism) diagram of transformation of functors from $D(X,\La)$ to $D(U^0,\La)$ as follows
\[
\begin{tikzpicture}[baseline=(current  bounding  box.center)]
\matrix(a)[matrix of math nodes, 
row sep=1.5em, column sep=2em, 
text height=1.5ex, text depth=0.45ex] 
{(f^0)^*\circ L^-_{X/S}&(f^0)^*\circ L^+_{X/S}\\ 
L^-_{U/S}\circ f^*& L^+_{U/S}\circ f^*,\\}; 
\path[->](a-1-1) edge node[above] {} (a-1-2);
\path[->](a-2-1) edge node[below] {} (a-2-2);
\path[->](a-1-1) edge node[left] {$\simeq$} (a-2-1);
\path[->](a-1-2) edge node[right] {$\simeq$}(a-2-2);
\end{tikzpicture}
\]
where the horizontal maps are constructed in \eqref{locmorph}.
\end{lem}

Again, let us consider the case of the natural transformations $(q^-)_*\to (i^-)^*$ and $(i^+)^!\to (q^+)_!$ first.

\begin{lem} \label{complemone}
Let $f\co U\to X$ as in Lemma \ref{etalelem}. There are commutative (up to natural isomorphism) diagrams of transformations of functors from $D(X^\pm,\La)$ to $D(U^0,\La)$ 
\[
\begin{tikzpicture}[baseline=(current  bounding  box.center)]
\matrix(a)[matrix of math nodes, 
row sep=1.5em, column sep=2em, 
text height=1.5ex, text depth=0.45ex] 
{(f^0)^*(q^-)_*&(f^0)^*(i^-)^* && (f^0)^*(i^+)^!&(f^0)^*(q^+)_!\\ 
(q^-)_*(f^-)^*&(i^-)^*(f^-)^*&& (i^+)^!(f^+)^*&(q^+)_!(f^+)^*,\\}; 
\path[->](a-1-1) edge node[above] {} (a-1-2);
\path[->](a-2-1) edge node[below] {} (a-2-2);
\path[->](a-1-1) edge node[left] {$\simeq$} (a-2-1);
\path[->](a-1-2) edge node[right] {$\simeq$}(a-2-2);
\path[->](a-1-4) edge node[above] {} (a-1-5);
\path[->](a-2-4) edge node[below] {} (a-2-5);
\path[->](a-1-4) edge node[left] {$\simeq$} (a-2-4);
\path[->](a-1-5) edge node[right] {$\simeq$}(a-2-5);
\end{tikzpicture}
\]
where the horizontal arrows are constructed from \eqref{contrafos}.
\end{lem}
\begin{proof} By Lemma \ref{bcforzero} we have $U^0=U^\pm\times_{X^\pm}X^0$ (use that $(X^\pm)^0=X^0$), and by Lemma \ref{bcforplus} we have $U^\pm=U^0\times_{X^0}X^\pm$. The diagrams of $S$-spaces in question are cartesian and we can use smooth base change to construct the vertical maps. Use that $f^0$ and $f^\pm$ are \'etale and hence, $(f^0)^*\simeq(f^0)^!$ and $(f^\pm)^*\simeq (f^\pm)^!$.
The commutativity of the functor diagrams is straight forward and left to the reader. 
\end{proof}

\begin{proof}[Proof of Lemma \ref{etalelem}.]  In view of Lemma \ref{complemone} it is enough to show that $f^*$ commutes with the transformation $(i^-)^*(p^-)^!\to (i^+)^!(p^+)^*$ in \eqref{locmorph}. There is a commutative diagram of $S$-spaces
\[
\begin{tikzpicture}[baseline=(current  bounding  box.center)]
\matrix(a)[matrix of math nodes, 
row sep=1em, column sep=1em, 
text height=1.5ex, text depth=0.45ex] 
{U^0&& U^- &\\
& X^0 && X^- \\
U^+&& U &\\ 
&X^+&& X, \\}; 
\path[->](a-1-1) edge node[above] {} (a-1-3) edge node[below] {} (a-2-2) edge (a-3-1);
\path[-stealth](a-1-3) edge (a-2-4) edge [densely dotted] (a-3-3);
\path[-stealth](a-3-1) edge node[left] {} (a-4-2) edge [densely dotted] (a-3-3);	
\path[->](a-2-2) edge (a-4-2) edge (a-2-4);						
\path[->](a-2-4) edge (a-4-4);
\path[->](a-4-2) edge (a-4-4);
\path[->](a-3-3) edge [densely dotted] node[left] {}(a-4-4);
\end{tikzpicture}
\]
where the $U$-square is cartesian, and $i^\pm\co X^0\to X^\pm$ factors through $j\co X^0\to X^+\times_XX^-$. The maps $f^0$ and $f^\pm$ are \'etale, cf. proof of Lemma \ref{complemone}. Let us explain how one checks commutativity of the diagram 
\[
\begin{tikzpicture}[baseline=(current  bounding  box.center)]
\matrix(a)[matrix of math nodes, 
row sep=1.5em, column sep=2em, 
text height=1.5ex, text depth=0.45ex] 
{(f^0)^*(i^-)^{*} (p^-)^!&(f^0)^*(i^+)^! (p^+)^*\\ 
(i^-)^{*} (p^-)^! f^*& (i^+)^! (p^+)^* f^*,\\}; 
\path[->](a-1-1) edge node[above] {} (a-1-2);
\path[->](a-2-1) edge node[below] {} (a-2-2);
\path[->](a-1-1) edge node[left] {$\simeq$} (a-2-1);
\path[->](a-1-2) edge node[right] {$\simeq$}(a-2-2);
\end{tikzpicture}
\]
where the vertical maps are constructed using $f^*\simeq f^!$ and the same for $f^0$ and $f^\pm$. Using the units $\id\to f_*f^*$ and $\id\to (p^+)_*(p^+)^*$ one constructs a commutative diagram
\[
\begin{tikzpicture}[baseline=(current  bounding  box.center)]
\matrix(a)[matrix of math nodes, 
row sep=1em, column sep=2em, 
text height=1.5ex, text depth=0.45ex] 
{\id & (p^+)_* (p^+)^*& \\ 
f_*f^*& f_*(p^+)_*(p^+)^*f^* & (p^+)_*(f^+)_*(f^+)^*(p^+)^*.\\}; 
\path[->](a-1-2) edge[bend left=10] node[above] {} (a-2-3);
\path[->](a-1-1) edge node[below] {} (a-1-2);
\path[->](a-1-1) edge node[left] {} (a-2-1);
\path[->](a-2-1) edge node[right] {}(a-2-2);
\path[->](a-2-2) edge node[above] {$\simeq$}(a-2-3);
\end{tikzpicture}
\]
Using adjunction for $f_*$ and applying $(i^-)^*(p^-)^!$ to the resulting diagram we get
\vspace{0.3cm}
\[
\begin{tikzpicture}[baseline=(current  bounding  box.center)]
\matrix(a)[matrix of math nodes, 
row sep=1em, column sep=2em, 
text height=1.5ex, text depth=0.45ex] 
{\overbrace{(i^-)^*(p^-)^!f^*}^{\simeq (f^0)^*(i^-)^*(p^-)^!} & \overbrace{(i^-)^*(p^-)^!f^*}^{\simeq (f^0)^*(i^-)^*(p^-)^!}(p^+)_*(p^+)^*& (f^0)^*(i^+)^!(p^+)^* \\ 
(i^-)^*(p^-)^!f^* & (i^-)^*(p^-)^!(p^+)_*(p^+)^*f^* & (i^+)^!(p^+)^*f^*,\\}; 
\path[->](a-1-1) edge node[below] {} (a-1-2);
\path[->](a-1-2) edge node[above] {$\psi$} (a-1-3);
\path[->](a-1-2) edge node[above] {} (a-2-2);
\path[->](a-1-1) edge node[left] {} (a-2-1);
\path[->](a-2-1) edge node[right] {}(a-2-2);
\path[->](a-2-2) edge node[above] {$\simeq$}(a-2-3);
\path[->](a-1-3) edge node[right] {$\simeq$}(a-2-3);
\end{tikzpicture}
\]
where $\psi$ is given by the $(j^!,j_*)$-adjunction. The composition of the arrows at the bottom (resp. the top) gives the desired map, and one checks that the right square commutes. This proves the lemma.
\end{proof}

\section{Functorial properties}
As a benefit of working over a general base scheme $S$, we are able to investigate the behaviour of hyperbolic localization with respect to base changes $S'\to S$. The situation is as good as one could hope. This is due to strong symmetry properties induced by the isomorphism in Braden's theorem.

\subsection{Base change}\label{bcforBraden}
Recall the following formalism. If $F,F',G,G':\calC\to \calD$ are functors between categories $\calC$ and $\calD$, and $\psi: F\to G$, $\phi:F'\to G'$ are natural transformations. Then a natural $2$-morphism $\Phi: \psi\Rightarrow \phi$ is a tuple $\Phi=(\Phi_F,\Phi_G)$ of natural transformations $\Phi_F:F\to F'$ and $\Phi_G:G\to G'$ such that the diagram
\[
\begin{tikzpicture}[baseline=(current  bounding  box.center)]
\matrix(a)[matrix of math nodes, 
row sep=1.5em, column sep=2em, 
text height=1.5ex, text depth=0.45ex] 
{F&F' \\ 
G&G' \\}; 
\path[->](a-1-1) edge node[above] {$\Phi_F$} (a-1-2);
\path[->](a-2-1) edge node[above] {$\Phi_G$} (a-2-2);
\path[->](a-1-1) edge node[left] {$\psi$} (a-2-1);
\path[->](a-1-2) edge node[right] {$\phi$} (a-2-2);
\end{tikzpicture}
\]
is commutative up to natural isomorphism. There is the obvious notion of a natural $2$-isomorphism. If $X/S$ is a space locally of finite presentation with an \'etale locally linearizable $\bbG_m$-action, then by \eqref{locmorph} above there is a natural transformation of functors on unbounded derived categories from $D(X,\La)$ to $D(X^0,\La)$ as follows
\begin{equation}\label{hallo}
\phi_{X}\co L^-_{X/S}\longto L^+_{X/S}.
\end{equation}
We abbreviate $D(X)=D(X,\La)$ and $D(X^0)=D(X^0,\La)$ in the following. 

\begin{prop}\label{funcprop}
Let $S$ be a scheme, and let $X/S$ be a space locally of finite presentation with an \'etale locally linearizable $\bbG_m$-action. Let $f\co X'\to X$ be a $\bbG_m$-equivariant morphism of $S$-spaces. Assume that for the hyperbolic localization diagrams, cf. Definition \ref{hyperbolicloc},
\[
\HypLoc{X'}\overset{\simeq}{\longto}\HypLoc{X}\times_XX'. 
\]
Let $f^0\co (X')^0\to X^0$ be the induced $S$-morphism on the spaces of fixed points. \smallskip\\
i) There are natural $2$-morphisms as follows.\smallskip\\
\phantom{h}\;\;(a) $\phi_X\circ f_*\Rightarrow f^0_*\circ \phi_{X'}$ as $2$-morphism of functors $D(X')\to D(X^0)$;\smallskip\\
\phantom{h}\;\;(b) $(f^0)^*\circ \phi_X\Rightarrow \phi_{X'}\circ f^*$ as $2$-morphism of functors $D(X)\to D((X')^0)$;\smallskip\\
ii) If $f$ is locally of finite presentation, then there are natural $2$-morphisms as follows.\smallskip\\
\phantom{h}\;\;(a) $f^0_!\circ \phi_{X'}\Rightarrow \phi_{X}\circ f_!$ as $2$-morphism of functors $D(X')\to D(X^0)$; \smallskip\\
\phantom{h}\;\;(b) $\phi_{X'}\circ f^!\Rightarrow (f^0)^!\circ \phi_X$ as $2$-morphism of functors $D(X)\to D((X')^0)$.\smallskip\\
iii) If $f$ is proper (resp. $f$ is smooth), then i).(a) and ii).(a) (resp. i).(b) and ii).(b)) are inverse to each other. \smallskip\\
iv) All transformations in i) and ii) restricted to the category $D^+(X)^{\Gmon}$ are natural 2-isomorphisms.
\end{prop}

\begin{rmk} Some sort of base change hypothesis seems to be necessary in order to construct the $2$-morphisms in i) and ii), cf. also Lemmas \ref{etalelem} and \ref{closedlem} above. Note that by Corollary \ref{bccor} the base change hypothesis on $X'\to X$ is satisfied if $X'=X\times S'$ for some morphism of schemes $S'\to S$. 
\end{rmk}

Recall the morphisms $i^\pm, q^\pm, p^\pm$ in the definition of $\HypLoc{X}$, cf. Definition \ref{hyperbolicloc}. Then 
\[
\phi_X\co L^-_{X/S}=(q^-)_*(p^-)^!\longto (q^+)_!(p^+)^*=L^+_{X/S}. 
\]
In the proof of the proposition, we use without further mentioning the following fact. The transformation $\phi_{X}$ can also be constructed as follows: Apply $(p^-)^!$ to the unit of the adjunction $\id \to (p^+)_*(p^+)^*$. Proceeding as in the construction of \eqref{locmorph}, we obtain a natural transformation
\begin{equation}\label{altercons}
(p^-)^!\longto (p^-)^!(p^+)_*(p^+)^*\longto (i^-)_*(i^+)^!(p^+)^*\longto (i^-)_*(q^+)_!(p^+)^*,
\end{equation}
where the last arrow comes from the transformation $(i^+)^!\to(q^+)_!$. Now apply $(q^-)_*$ to \eqref{altercons} and use $q^-\circ i^-=\id$ to obtain a transformation $(q^-)_*(p^-)^!\to (q^+)_!(p^+)^*$ which agrees up to natural isomorphism with $\phi_X$.

\begin{proof}[Proof of Proposition \ref{funcprop}] Let us denote the corresponding morphisms for $\HypLoc{X'}$ also by $i^\pm, q^\pm, p^\pm$ (by abuse of notation). For i).(a), we have to construct natural transformations $\Phi^+$ and $\Phi^-$ of functors
\begin{equation}\label{fundiagD1}
\begin{tikzpicture}[baseline=(current  bounding  box.center)]
\matrix(a)[matrix of math nodes, 
row sep=2em, column sep=2em, 
text height=1.5ex, text depth=0.45ex] 
{(q^-)_{*}(p^-)^!f_*& (f^0)_*(q^{-})_{*}(p^{-})^! \\ 
(q^+)_{!}(p^+)^*f_*& (f^0)_*(q^{+})_{!}(p^{+})^*\\}; 
\path[->](a-1-1) edge node[above] {$\Phi^-$} (a-1-2);
\path[->](a-2-1) edge node[above] {$\Phi^+$} (a-2-2);
\path[->](a-1-1) edge node[left] {$\phi_X f_*$} (a-2-1);
\path[->](a-1-2) edge node[right] {$(f^0)_* \phi_{X'}$} (a-2-2);
\end{tikzpicture}
\end{equation}
such that \eqref{fundiagD1} commutes up to natural isomorphism. By assumption both squares in the diagram
\begin{equation}\label{doublecart}
\begin{tikzpicture}[baseline=(current  bounding  box.center)]
\matrix(a)[matrix of math nodes, 
row sep=2em, column sep=2em, 
text height=1.5ex, text depth=0.45ex] 
{(X')^0& (X')^\pm & X' \\ 
X^0& X^\pm & X\\}; 
\path[->](a-1-2) edge node[above] {$q^\pm$} (a-1-1);
\path[->](a-1-2) edge node[above] {$p^\pm$} (a-1-3);
\path[->](a-2-2) edge node[below] {$q^\pm$} (a-2-1);
\path[->](a-2-2) edge node[below] {$p^\pm$} (a-2-3);
\path[->](a-1-1) edge node[right] {$f^0$} (a-2-1);
\path[->](a-1-2) edge node[right] {$f^\pm$} (a-2-2);
\path[->](a-1-3) edge node[right] {$f$} (a-2-3);
\end{tikzpicture}
\end{equation}
are cartesian. \smallskip\\
\emph{Construction of $\Phi^-$:}
Base change applied to the right square in \eqref{doublecart} gives a natural isomorphism $(p^-)^!f_*\simeq f^-_{*} ({p^-})^!$. Apply $(q^-)_{*}$ to this isomorphism, and use that the left square in \eqref{doublecart} commutes. This constructs the natural isomorphism $\Phi^-\co (q^-)_{*}(p^-)^!f_*\to f^0_*(q^-)_{*}(p^-)^!$.\smallskip\\
\emph{Construction of $\Phi^+$:} Applying $(p^+)^*$ to the adjuction $f^*f_*\to \id$ gives
\[
(f^+)^*(p^+)^*f_*\simeq (p^+)^*f^*f_*\longto (p^+)^*.
\]
By adjunction, we obtain a transformation $(p^+)^*f_*\to (f^+)_*(p^+)^*$. Applying $(q^+)_{!}$ it remains to construct a transformation
\begin{equation}\label{lasteq}
(q^+)_{!}(f^+)_{*}\longto (f^0)_*(q^+)_{!}.
\end{equation}
Applying $(q^+)_{!}$ to the adjunction $(f^+)^*(f^+)_{*}\to \id$ gives
\[
(f^0)^*(q^+)_{!}(f^+)_{*}\simeq (q^+)_{!}(f^+)^*(f^+)_{*}\longto (q^+)_{!},
\]
where the isomorphism follows from the base change theorem applied to the left cartesian square in \eqref{doublecart}. Using adjunction this concludes the construction of \eqref{lasteq}, and hence the construction of $\Phi^+\co (q^+)_{!}(p^+)^*f_*\to (f^0)_*(q^+)_{!}(p^+)^*$. \smallskip\\
\emph{Diagram \eqref{fundiagD1} commutes up to natural isomorphism:} We claim that it is enough to check the commutativity of the following diagrams, whose construction is explained below. Each isomorphism in (C1)-(C3) below is deduced by base change using our assumption $\HypLoc{X'}=\HypLoc{X}\times_XX'$. \smallskip\\
\emph{Compatibility 1 (C1):}
\[
\begin{tikzpicture}[baseline=(current  bounding  box.center)]
\matrix(a)[matrix of math nodes, 
row sep=1.5em, column sep=2em, 
text height=1.5ex, text depth=0.45ex] 
{f_* & f_* & \\ 
(p^+)_{*}(p^+)^*f_* & f_*(p^+)_{*}(p^+)^*  &(p^+)_{*}(f^+)_{*}(p^+)^*\\}; 
\path[->](a-1-1) edge  (a-2-1);
\path[->](a-2-1) edge  (a-2-2);
\path[->](a-1-1) edge node[above] {$\id$} (a-1-2);
\path[->](a-1-2) edge (a-2-2);
\path[->](a-2-2) edge node[above] {$\simeq$} (a-2-3);
\end{tikzpicture}
\]
\emph{Compatibility 2 (C2):}
\[
\begin{tikzpicture}[baseline=(current  bounding  box.center)]
\matrix(a)[matrix of math nodes, 
row sep=1.5em, column sep=2em, 
text height=1.5ex, text depth=0.45ex] 
{(p^-)^!(p^+)_{*}(f^+)_{*} & (f^-)_{*}(p^-)^! (p^+)_{*}&\\ 
(i^-)_{*}(i^+)^!(f^+)_{*} & (f^-)_{*}(i^-)_{*}(i^+)^! & (i^-)_{*}(i^+)^!(f^+)_* \\}; 
\path[->](a-1-1) edge  (a-2-1);
\path[->](a-2-1) edge node[above] {$\simeq$}  (a-2-2);
\path[->](a-1-1) edge node[above] {$\simeq$} (a-1-2);
\path[->](a-1-2) edge (a-2-2);
\path[->](a-2-2) edge node[above] {$\simeq$} (a-2-3);
\end{tikzpicture}
\]
\emph{Compatibility 3 (C3):}
\[
\begin{tikzpicture}[baseline=(current  bounding  box.center)]
\matrix(a)[matrix of math nodes, 
row sep=1.5em, column sep=2em, 
text height=1.5ex, text depth=0.45ex] 
{(i^+)^!(f^+)_{*}& (f^0)_*(i^+)^! \\ 
(q^+)_{!}(f^+)_{*}& (f^0)_*(q^+)_{!}\\}; 
\path[->](a-1-1) edge (a-2-1);
\path[->](a-2-1) edge (a-2-2);
\path[->](a-1-1) edge node[above] {$\simeq$} (a-1-2);
\path[->](a-1-2) edge (a-2-2);
\end{tikzpicture}
\]
We give the recipe how (C1)-(C3) imply the commutativity of \eqref{fundiagD1}\footnote{The author recommends a big sheet of paper to check the commutativity.}. Apply $(p^-)^!$ from the left to (C1). Using base change, the upper right of (C1) may be replaced by $(f^-)_{*}(p^-)^!$. Next apply (C2) to the lower right, and extend the resulting diagram at the very left by the commutative diagram
\[
\begin{tikzpicture}[baseline=(current  bounding  box.center)]
\matrix(a)[matrix of math nodes, 
row sep=1.5em, column sep=2em, 
text height=1.5ex, text depth=0.45ex] 
{(p^-)^!(p^+)_{*}(p^+)^*f_*& (p^-)^!(p^+)_{*}(f^+)_{*}(p^+)^*\\ 
(i^-)_{*}(i^+)^!(p^+)^*f_*& (i^-)_{*}(i^+)^!(f^+)_{*}(p^+)^*,\\}; 
\path[->](a-1-1) edge (a-2-1);
\path[->](a-2-1) edge (a-2-2);
\path[->](a-1-1) edge (a-1-2);
\path[->](a-1-2) edge (a-2-2);
\end{tikzpicture}
\]
which is derived from $(p^+)^*f_*\to (f^+)_{*}(p^+)^*$ (cf. the construction of $\Phi^+$) and $(p^-)^!(p^+)_{*}\to (i^-)_{*}(i^+)^!$ (cf. the middle arrow in \eqref{compD1}). Now apply $(q^-)_{*}$ to everything, and use $(q^-)_{*}(i^-)_{*}=\id$. By base change, the morphism at the top is $\Phi^-$. At the lower right apply (C3), and extend the resulting diagram at the very left by the commutative diagram
\[
\begin{tikzpicture}[baseline=(current  bounding  box.center)]
\matrix(a)[matrix of math nodes, 
row sep=1.5em, column sep=2em, 
text height=1.5ex, text depth=0.45ex] 
{(i^+)^!(p^+)^*f_*& (i^+)^!(f^+)_*(p^+)^*\\ 
(q^+)_{!}(p^+)^*f_*& (q^+)_{!}(f^+)_*(p^+)^*.\\}; 
\path[->](a-1-1) edge (a-2-1);
\path[->](a-2-1) edge (a-2-2);
\path[->](a-1-1) edge (a-1-2);
\path[->](a-1-2) edge (a-2-2);
\end{tikzpicture}
\]
The morphism at the bottom is $\Phi^+$. This implies the commutativity of \eqref{fundiagD1}. It remains to show (C1)-(C3).\smallskip\\
\emph{Proof of (C1):} There is a commutative diagram
\[
\begin{tikzpicture}[baseline=(current  bounding  box.center)]
\matrix(a)[matrix of math nodes, 
row sep=1.5em, column sep=2em, 
text height=1.5ex, text depth=0.45ex] 
{(p^+)^*f_* & (p^+)^*f_*\\ 
(p^+)^*f_* & (f^+)_*(p^+)_*,\\}; 
\path[->](a-1-1) edge  node[left] {$\id$} (a-2-1);
\path[->](a-2-1) edge  (a-2-2);
\path[->](a-1-1) edge node[above] {$\id$} (a-1-2);
\path[->](a-1-2) edge (a-2-2);
\end{tikzpicture}
\]
where $(p^+)^*f_*\to (f^+)_*(p^+)_*$ is the morphism defined in the construction of $\Phi^+$ above. By $((p^+)^*,(p^+)_*)$-adjunction applied to the vertical arrows and using $(p^+)_*(f^+)_*=f_*(p^+)_*$ (at the lower right), we obtain (C1). \smallskip\\
\emph{Proof of (C2):} The vertical arrows in (C2) are constructed as in \eqref{compD1} above using the $(j^!,j_*)$-adjunction. Then (C2) follows by base change from the fact that the following diagram
\[
\begin{tikzpicture}[baseline=(current  bounding  box.center)]
\matrix(a)[matrix of math nodes, 
row sep=1.5em, column sep=2em, 
text height=1.5ex, text depth=0.45ex] 
{(X')^0& (X')^+\times_{X'}(X')^- & X' \\ 
X^0& X^+\times_{X}X^- & X\\}; 
\path[->](a-1-1) edge (a-1-2);
\path[->](a-1-2) edge (a-1-3);
\path[->](a-2-1) edge (a-2-2);
\path[->](a-2-2) edge (a-2-3);
\path[->](a-1-1) edge node[right] {$f^0$} (a-2-1);
\path[->](a-1-2) edge  (a-2-2);
\path[->](a-1-3) edge node[right] {$f$} (a-2-3);
\end{tikzpicture}
\]
is cartesian.\smallskip\\
\emph{Proof of (C3):} Using the adjunctions $(i^+)_!(i^+)^!\to \id$ and $(f^+)^*(f^+)_*\to \id$, one constructs a commutative diagram 
\[
\begin{tikzpicture}[baseline=(current  bounding  box.center)]
\matrix(a)[matrix of math nodes, 
row sep=1.5em, column sep=2em, 
text height=1.5ex, text depth=0.45ex] 
{(f^+)^*(f^+)_*(i^+)_!(i^+)^! &  (i^+)_!(i^+)^!\\\ 
(f^+)^*(f^+)_*& \id.\\}; 
\path[->](a-1-1) edge (a-2-1);
\path[->](a-2-1) edge (a-2-2);
\path[->](a-1-1) edge (a-1-2);
\path[->](a-1-2) edge (a-2-2);
\end{tikzpicture}
\]
Now apply $(q^+)_!$ to the diagram. Using base change (at the left), additionally $(i^+)_!=(i^+)_*$ (at the upper left) and $(q^+)_!(i^+)_!=\id$ (at the top), the diagram becomes
\[
\begin{tikzpicture}[baseline=(current  bounding  box.center)]
\matrix(a)[matrix of math nodes, 
row sep=1.5em, column sep=2em, 
text height=1.5ex, text depth=0.45ex] 
{(f^0)^*(i^+)^!(f^+)_{*}& (i^+)^! \\ 
(f^0)^*(q^+)_{!}(f^+)_{*}& (q^+)_{!},\\}; 
\path[->](a-1-1) edge (a-2-1);
\path[->](a-2-1) edge (a-2-2);
\path[->](a-1-1) edge (a-1-2);
\path[->](a-1-2) edge (a-2-2);
\end{tikzpicture}
\]
and we obtain (C3) by $((f^0)^*, (f^0)_*)$-adjunction. The bottom is the natural transformation \eqref{lasteq} and the vertical arrows are deduced from $(i^+)^!\to (q^+)_{!}$.

This proves (C1)-(C3), and hence part i).(a). The $2$-morphism in part i).(b) is constructed from part i).(a) as follows. By adjunction we get a $2$-morphism $(f^0)^*\phi_X f_*\Rightarrow \phi_{X'}$. Applying $f^*$ from the right gives $(f^0)^*\phi_X f_*f^*\Rightarrow \phi_{X'}f^*$. Now define the $2$-morphism in part i).(b) as the composition
\[
(f^0)^*\phi_X\Rightarrow(f^0)^*\phi_X f_*f^*\Rightarrow \phi_{X'}f^*,
\]
where the first arrow is deduced from the adjunction morphism $\id\to f_*f^*$. Note that directly constructing $(f^0)^*\phi_X\Rightarrow \phi_{X'}f^*$ results in the same $2$-morphism. This shows part i).\smallskip\\ 
Part ii).(a) follows from part i).(a) by formally interchanging all $*$ with $!$, inverting all arrows and interchanging all $+$ with $-$. Part ii).(b) follows again from part ii).(a) by adjunction. This shows part ii).\smallskip\\
Now if $f$ is proper (resp. smooth), then $f^0$ and $f^\pm$ are proper (resp. smooth) by base change. In this case, both $\Phi^+$ and $\Phi^-$ are deduced from proper (resp. smooth) base change and hence are isomorphisms and the corresponding transformations are inverse to each other. This shows part iii).\smallskip\\
Let $\calA$ be a $\bbG_m$-monodromic bounded below complex. Then $f_*(\calA)$ is $\bbG_m$-monodromic, and both transformations $\phi_Xf_*(\calA)$ and $(f^0)_*\phi_{X'}(\calA)$ in \eqref{fundiagD1} are isomorphisms. The transformation $\Phi^-$ in \eqref{fundiagD1} is deduced from base change, hence an isomorphism. Then three morphisms in \eqref{fundiagD1} are isomorphisms, and $\Phi^+$ needs to be an isomorphism as well. The cases of $f^*$ and $f_!$, $f^!$ are proven similarly: in the suitable diagrams three out of four transformations are isomorphisms and hence the remaining transformation needs to be an isomorphism. This implies part iv) and the proposition follows.
\end{proof}

\subsection{Commutation with nearby cycles}\label{nbhccommute}
Let $\calO$ be a henselian discrete valuation ring with field of fractions $F$ and residue field $k$. Let $\sF$ be a separable closure of $F$, and denote by $\bar{\calO}$ the integral closure of $\calO$ in $\sF$. Let $\bk$ be the residue field of $\bar{\calO}$ (which is a separable closure of $k$). 
Let $S=\Spec(\calO)$, $s=\Spec(k)$, $\eta=\Spec(F)$, $\bar{S}=\Spec(\bar{\calO})$, $\bar{s}=\Spec(\bk)$, $\bar{\eta}=\Spec(\sF)$. This gives the $7$-tuple $(S,s,\eta,\bar{S},\bar{s},\bar{\eta},\Ga)$, where $\Ga=\Gal(\sF/F)$ is the Galois group. For a space $X/S$, there is a commutative diagram
\[
\begin{tikzpicture}[baseline=(current  bounding  box.center)]
\matrix(a)[matrix of math nodes, 
row sep=1.5em, column sep=2em, 
text height=1.5ex, text depth=0.45ex] 
{X_{\bar{\eta}}& X_{\bar{S}} & X_{\bar{s}} &\text{cartesian above}& \bar{\eta} & \bar{S} &\bar{s}\\ 
X_\eta& X & X_s &&\eta & S & s.\\}; 
\path[->](a-1-1) edge node[above] {$\bar{j}$} (a-1-2)
			edge (a-2-1);
\path[->](a-2-1) edge node[above] {$j$} (a-2-2);
\path[->](a-1-3) edge node[above] {$\bar{i}$} (a-1-2);
\path[->](a-1-2) edge (a-2-2);
\path[->](a-1-3) edge (a-2-3);
\path[->](a-2-3) edge node[above] {$i$} (a-2-2);

\path[->](a-1-5) edge (a-1-6)
			edge (a-2-5);
\path[->](a-2-5) edge (a-2-6);
\path[->](a-1-7) edge (a-1-6);
\path[->](a-1-6) edge (a-2-6);
\path[->](a-1-7) edge (a-2-7);
\path[->](a-2-7) edge (a-2-6);
\end{tikzpicture}
\]
Let $\La=\bbZ/n$ with $n>1$ invertible on $S$. By [SGA 7, XIII], there is the functor of nearby cycles
\begin{align*}
\Psi_X\co D(X_\eta,\La)&\longto D(X_s\times_S\eta,\La),\\
\calA &\longmapsto \bar{i}^*\bar{j}_*\calA_{\bar{\eta}}
\end{align*}
where $\bar{j}_*$ denotes the derived push forward, and $D(X_s\times_S\eta,\La)$ is as in [SGA 7, XIII] the derived category of $((X_{\bar{s}})_{\text{\'et}}, \La)$-modules with a continuous action of $\Ga$ compatible its action on $X_{\bar{s}}$. For a morphism $S$-morphism $f\co Y\to X$ we get functors $f_*$, $f^*$ and $f_!$, $f^!$ (if $f$ is locally of finite type) on the category $D(X_s\times_S\eta,\La)$ as in [SGA 7, XIII 2.1.6, 2.1.7] satisfying the usual adjointness properties and functorialities with respect to $\Psi$.

If $X/S$ is equipped with a $\bbG_{m}$-action, then, for any $S'\to S$, the scheme $X_{S'}$ is equipped with the induced $\bbG_{m}$-action (by base change). If the action is \'etale locally linearizable, there is for any $\calA\in D(X_{S'},\La)$ the arrow of $D(X_{S'}^0,\La)$ defined in \eqref{locmorph}
\[
L_{S'}^-\calA\longto L_{S'}^+\calA,
\]
where $L_{S'}^\pm=L_{X_{S'}/{S'}}^\pm$ denote the hyperbolic localization functors, cf. Definition \ref{hyplocfun}.  

\begin{thm}\label{nbhcthm} Let $S$ be the spectrum of a henselian discrete valuation ring, and let $X/S$ be a space of finite type with an \'etale locally linearizable $\bbG_{m}$-action. Then, for $\calA\in D^+(X_\eta,\La)$, there is a commutative diagram of arrows in $D(X_s\times_S\eta,\La)$ 
\[
\begin{tikzpicture}[baseline=(current  bounding  box.center)]
\matrix(a)[matrix of math nodes, 
row sep=1.5em, column sep=2em, 
text height=1.5ex, text depth=0.45ex] 
{ L^-_{\bar{s}}\circ\Psi_X(\calA)& \Psi_{X^0}\circ L^-_\eta(\calA)\\ 
L^+_{\bar{s}}\circ \Psi_X(\calA)&\Psi_{X^0}\circ L^+_\eta(\calA),\\ }; 
\path[->](a-1-2) edge (a-1-1); 
\path[->](a-1-1) edge  (a-2-1); 
\path[->](a-2-1) edge (a-2-2);
\path[->](a-1-2) edge  (a-2-2);
\end{tikzpicture}
\]
and all arrows are isomorphisms if $\calA$ is $\bbG_m$-monodromic. 
\end{thm}

\begin{rmk} The transformation $L^-_{\bar{s}}\to L^+_{\bar{s}}$ is defined on the category $D(X_s\times_S\eta,\La)$ since its construction in \eqref{locmorph} is purely formal using adjointness properties.
\end{rmk}

\begin{proof} By Proposition \ref{funcprop} i), there is a commutative (up to natural isomorphism) diagram of transformations
\[
\begin{tikzpicture}[baseline=(current  bounding  box.center)]
\matrix(a)[matrix of math nodes, 
row sep=1.5em, column sep=2em, 
text height=1.5ex, text depth=0.45ex] 
{ L^-_{\bar{S}} \bar{j}_*& \bar{j}^0_* L^-_{\bar{\eta}} \\ 
L^+_{\bar{S}} \bar{j}_* & \bar{j}^0_* L^+_{\bar{\eta}}.\\ }; 
\path[->](a-1-1) edge (a-1-2); 
\path[->](a-1-1) edge  (a-2-1); 
\path[->](a-2-1) edge (a-2-2);
\path[->](a-1-2) edge  (a-2-2);
\end{tikzpicture}
\]
Applying $(\bar{i}^0)^*$ and using Proposition \ref{funcprop} iii), we get a commutative (up to natural isomorphism) diagram of transformations
\begin{equation}\label{nearbycomm}
\begin{tikzpicture}[baseline=(current  bounding  box.center)]
\matrix(a)[matrix of math nodes, 
row sep=1.5em, column sep=2em, 
text height=1.5ex, text depth=0.45ex] 
{L^-_{\bar{s}} \bar{i}^*\bar{j}_* & (\bar{i}^0)^* L^-_{\bar{S}} \bar{j}_* & (\bar{i}^0)^*\bar{j}^0_* L^-_{\bar{\eta}} \\ 
L^+_{\bar{s}} \bar{i}^*\bar{j}_* & (\bar{i}^0)^*L^+_{\bar{S}} \bar{j}_* & (\bar{i}^0)^* \bar{j}^0_* L^+_{\bar{\eta}}.\\ }; 
\path[->](a-1-1) edge (a-2-1);
\path[->](a-1-2) edge (a-1-1);
\path[->](a-2-2) edge node[above] {$\simeq$} (a-2-1);
\path[->](a-1-2) edge node[above] {$\simeq$} (a-1-3); 
\path[->](a-1-2) edge  (a-2-2); 
\path[->](a-2-2) edge (a-2-3);
\path[->](a-1-3) edge  (a-2-3);
\end{tikzpicture}
\end{equation}
Using Proposition \ref{funcprop} iii) and a limit argument we see that $L^\pm_{\bar{\eta}}\circ (\str)_{\bar{\eta}}\overset{\simeq}{\to}(\str)_{\bar{\eta}}\circ L^\pm_{\eta}$. This concludes the construction of the diagram above. By construction the transformations agree with the ones coming from the functorialities of the nearby cycles.\\
Now if $\calA$ is $\bbG_m$-monodromic, then by Lemma \ref{funcmonlem}, the complexes $\calA_{\bar{\eta}}$, $\bar{j}_*\calA_{\bar{\eta}}$ and $\Psi_X\calA$ are $\bbG_m$-monodromic and hence by Theorem \ref{Bradenthm}, all vertical arrows in \eqref{nearbycomm} are isomorphisms. This concludes the proof of the theorem. \end{proof}

\begin{ex} \label{p1example} Let $S=\Spec(\bbZ_p)$, and let $X$ be the flat projective $\bbZ_p$-scheme such that $X_\eta=\bbP^1_\eta$ and such that $X_s$ is the intersection of two $\bbP^1_s$'s meeting transversally at a single $s$-point $e_s$. The scheme $X$ is equipped with a $\bbG_{m}$-action inducing the usual action on $\bbP^1_\eta$. The $\bbQ_p$-points $0_\eta$ and $\infty_\eta$ extend by properness to $\bbZ_p$-points $0_S$ and $\infty_S$ which are fixed under the $\bbG_{m}$-action. Then on reduced loci $X^0=0_S\amalg \infty_S \amalg e_s$ is the subscheme of fixed points. The attractor (resp. repeller) is on reduced loci
\begin{equation}\label{exattrep}
X^+= (\bbA^1_S)^+\amalg \infty_S\amalg (\bbA^1_s)^+\;\;\;\;\text{(resp. $X^-=0_S\amalg (\bbA^1_S)^-\amalg (\bbA^1_s)^-$).}
\end{equation}
The maps $p^{\pm}\co X^\pm\to X$ are monomorphisms (because $X$ is separated) such that on intersections 
\[
X^+\times_XX^-=X^0\amalg \bbG_{m,\eta}\amalg \bbG_{m,s}\amalg \bbG_{m,s}.
\]
The morphisms $q^{\pm}\co X^\pm\to X^0$ are given by contracting \eqref{exattrep} to the fixed points. The complex $\calA=\bbZ/n\langle1\rangle$ on $X_\eta$ is $\bbG_{m}$-monodromic, and one computes for the hyperbolic localization
\[
L_{X_\eta/\eta}^{\pm}(\calA)=\bbZ/n\langle-1\rangle\oplus \bbZ/n\langle 1\rangle.
\] 
The nearby cycles $\Psi_{X^0}$ are constant, i.e. $\Psi_{X^0}\circ L_{X_\eta/\eta}^{\pm}(\calA)=\bbZ/n\langle-1\rangle\oplus \bbZ/n\langle 1\rangle$. Hence, Theorem \ref{nbhcthm} implies on compact cohomology
\[
R\Ga_c(X_s^+,\Psi_X(\calA))= \bbZ/n\langle-1\rangle\oplus \bbZ/n\langle 1\rangle.
\]
i.e. $R\Ga_c((\bbA_s^1)^+,\Psi_X(\calA))$ is $\bbZ/n\langle-1\rangle$ (resp. $0$) on the flat (resp. non-flat) copy of $\bbA^1$ in \eqref{exattrep}. 
\end{ex}


\end{document}